\documentclass[12pt]{article}

\usepackage{amsmath}\usepackage{amssymb}\usepackage{amsfonts}\usepackage{amsthm}
\usepackage{color,url}
\usepackage{algorithm}
\usepackage{algorithmic}
\usepackage{dsfont}
\usepackage{setspace}
\usepackage{graphicx}
\usepackage[english]{babel}\usepackage[latin1]{inputenc}\usepackage{times}\usepackage[T1]{fontenc}
\usepackage{epstopdf}
\usepackage{epsfig}
\usepackage{float}
\usepackage{stmaryrd}

 \newtheorem{thm}{Theorem}[section]
 \newtheorem{cor}[thm]{Corollary}
 \newtheorem{lem}[thm]{Lemma}
 \newtheorem{prop}[thm]{Proposition}
 \theoremstyle{definition}
 \newtheorem{defn}[thm]{Definition}
 \theoremstyle{remark}
 \newtheorem{rem}[thm]{Remark}
 \newtheorem{ex}[thm]{Example}
 
 \newtheorem{Fact}[thm]{Fact}
 \numberwithin{equation}{section}

\newcommand{\dom}{\operatorname{dom}}
\newcommand{\intt}{\operatorname{int}}

\newcommand{\co}{\operatorname{co}}

\newcommand{\ran}{\operatorname{ran}}
\newcommand{\sgn}{\operatorname{sgn}}
\newcommand{\Id}{\operatorname{Id}}

\newcommand{\gra}{\operatorname{gra}}

\newcommand{\argmin}{\operatornamewithlimits{argmin}}

\newcommand{\R}{\operatorname{\mathbb{R}}}
\newcommand{\N}{\operatorname{\mathbb{N}}}

\newcommand{\ORR}{\operatorname{\overline{\mathbb{R}}}}

\begin{document}
\title{Proximal mappings and Moreau envelopes of single-variable convex piecewise cubic functions and multivariable gauge functions}
\author{C. Planiden\thanks{Mathematics and Applied Statistics, University of Wollongong, Wollongong, NSW, 2522, Australia. chayne@uow.edu.au} \and X. Wang\thanks{Mathematics, University of British Columbia Okanagan, Kelowna, B.C., V1V 1V7, Canada. shawn.wang@ubc.ca}}
\date{January 24, 2019}
\maketitle
\begin{abstract}
This work presents a collection of useful properties of the Moreau envelope for finite-dimensional, proper, lower semicontinuous, convex functions. In particular, gauge functions and piecewise cubic functions are investigated and their Moreau envelopes categorized. Characterizations of convex Moreau envelopes are established; topics include strict convexity, strong convexity and Lipschitz continuity.
\end{abstract}
\noindent\textbf{AMS Subject Classification.}\medskip\\
Primary 49J53, 52A41; Secondary 49J50, 26C05\medskip\\
\noindent\textbf{Keywords.}\medskip\\
Gauge function, Lipschitz continuous, Moreau envelope, piecewise cubic function, proximal mapping, strictly convex, strongly convex\medskip\\

\section{Introduction}\label{sec:intro}

The Moreau envelope $e_rf$ was introduced in its original form by Jean-Jacques Moreau in the mid-1960s \cite{proximite}. It is an infimal convolution of two functions $f$ and $q_r,$ where $r>0$ and $q_r=\frac{r}{2}\|\cdot\|^2$. The Moreau envelope offers many benefits in Optimization, such as the smoothing of the nonsmooth objective function $f$ \cite{moreau1963,proximite} while maintaining the same minimum and minimizers of $f$ in the case where $f$ is proper, lower semicontinuous (lsc) and convex \cite{rockwets,funcanal}. Also in the convex setting, $e_rf$ is differentiable and its gradient has an explicit representation, even when $f$ itself is not differentiable \cite{rockwets}. As a result, much research has been done on properties of the Moreau envelope, including differentiability \cite{chen2012themoreau,diffprop,genhess}, regularization \cite{bayen2016about,hinter2009moreau,jourani2017moreau,keuthen2015moreau,practasp,mifflin1999properties} and convergence of the related proximal-point algorithms for finding a minimizer \cite{aragon2007convergence,bento2014finite,hare2014derivative,xiao2014proximal,zaslavski2010convergence}.\par In this work, we continue the development of convex Moreau envelope theory. We endeavour to show the advantages that this form of regularization has to offer and make comparisons to the Pasch-Hausdorff envelope. Most of the focus is on the set of convex Moreau envelopes; we work to establish characterizations about when a function is a Moreau envelope of a proper, lsc, convex function. We also consider the differentiability properties of $e_rf,$ annotating the characteristics of the proximal mapping and the Moreau envelope for $\mathcal{C}^k$ functions.\par The main contributions of this paper are the analysis of the proximal mapping and Moreau envelope of two particular families of convex functions: piecewise cubic functions and gauge functions. Explicit formulae for, and method of calculation of, the proximal mapping and Moreau envelope for any single-variable convex piecewise-cubic function are given. The Moreau envelope is used to smooth the ridges of nondifferentiability (except the kernel) of any gauge function, while maintaining its status as a gauge function. The special case of norm functions is analyzed as well; the Moreau envelope is used to convert any norm function into one that is smooth everywhere expect at the origin. For both piecewise-cubic functions and gauge functions, several explicit examples with illustrations are included. To the best of our knowledge, the closed forms of Moreau envelopes of many examples given here have not been realized until now. The piecewise cubic work extends the results for piecewise linear-quadratic functions found in \cite{bajaj2017visualization,convhullalg,plqmodel}, and the study of Moreau envelopes of gauge functions is new. \par The remainder of this paper is organized as follows. Section \ref{sec:prelim} contains notation, definitions and facts that are used throughout. In Section \ref{sec:props}, several known results about the Moreau envelope are collected first, then new results on the set of convex Moreau envelopes are presented. We provide an upper bound for the difference $f-e_rf$ when $f$ is a Lipschitz continuous function. We establish several characterizations of the Moreau envelope of $f$ convex, based on strict convexity of $f,$ strong convexity of the Fenchel conjugate $(e_rf)^*$ and Lipschitz continuity of $\nabla e_rf.$ We discuss the differentiability of $e_rf,$ proving that $f\in\mathcal{C}^k\Rightarrow e_rf\in\mathcal{C}^k.$ Then we focus on explicit expressions for the Moreau envelope and the proximal mapping for convex piecewise functions on $\R,$ which sets the stage for the section that follows. Section \ref{sec:pcf} concentrates on the set of convex piecewise-cubic functions on $\R$ and their Moreau envelopes. We lay out the piecewise domain of $e_rf$ for $f$ piecewise-cubic and present a theorem that states the proximal mapping and Moreau envelope. Section \ref{sec:anf} deals with the smoothing of an arbitrary gauge function by way of the Moreau envelope. It is shown that given a gauge function $f,$ the function $\sqrt{e_r(f^2)}$ is also a gauge function and is differentiable everywhere except on the kernel. A corollary about norm functions follows; if $f$ is a norm function, then $\sqrt{e_r(f^2)}$ is a norm function that is differentiable everywhere except at the origin. Several examples and illustrations are provided in this section. Section \ref{sec:conc} summarizes the results of this work.

\section{Preliminaries}\label{sec:prelim}

\subsection{Notation}

All functions in this work are defined on $\R^n,$ Euclidean space equipped with inner product defined $\langle x,y\rangle=\sum_{i=1}^nx_iy_i$ and induced norm $\|x\|=\sqrt{\langle x,x\rangle}.$ The extended real line $\R\cup\{\infty\}$ is denoted $\overline{\R}.$ We use $\Gamma_0(\R^n)$ to represent the set of proper, convex, lower semicontinuous (lsc) functions on $\R^n.$ The identity operator is denoted $\Id.$ We use $N_C(x)$ to represent the normal cone to $C$ at $x,$ as defined in \cite{rockwets}. The domain and the range of an  operator $A$ are denoted $\dom A$ and $\ran A,$ respectively. Pointwise convergence is denoted $\overset{p}\to,$ epiconvergence $\overset{e}\to.$

\subsection{Definitions and facts}

In this section, we collect some definitions and facts that we need for proof of the main results.
\begin{defn}
The \emph{graph} of an operator $A:\R^n\rightrightarrows\R^n$ is defined
$$\gra A=\{(x,x^*):x^*\in Ax\}.$$Its inverse $A^{-1}:\R^n\rightrightarrows\R^n$ is defined by the graph
$$\gra A^{-1}=\{(x^*,x):x^*\in Ax\}.$$
\end{defn}
\begin{defn}
For any function $f:\R^n\rightarrow\ORR,$ the \emph{Fenchel conjugate} of $f$ is denoted $f^*:\R^n\rightarrow\ORR$ and defined by
$$f^*(x^*)=\sup\limits_{x\in\R^n}[\langle x^*,x\rangle-f(x)].$$
\end{defn}
\begin{defn}\label{df:Moreau}
For a proper, lsc function $f:\R^n\rightarrow\ORR,$ the \emph{Moreau envelope} of $f$ is denoted $e_rf$
and defined by
$$e_rf(x)=\inf\limits_{y\in\R^n}\left\{f(y)+\frac{r}{2}\|y-x\|^2\right\}.$$
The vector $x$ is called the \emph{prox-centre} and the scalar $r\geq0$ is called the \emph{prox-parameter}. The associated \emph{proximal mapping} is the set of all points at which the above infimum is attained, denoted $P_rf:$
$$P_rf(x)=\argmin\limits_{y\in\R^n}\left\{f(y)+\frac{r}{2}\|y-x\|^2\right\}.$$
\end{defn}
\begin{defn}
A function $f\in\Gamma_0(\R^n)$ is \emph{$\sigma$-strongly convex} if there exists a modulus $\sigma>0$ such that $f-\frac{\sigma}{2}\|\cdot\|^2$ is convex. Equivalently, $f$ is $\sigma$-strongly convex if there exists $\sigma>0$ such that for all $\lambda\in(0,1)$ and for all $x,y\in\R^n,$
$$f(\lambda x+(1-\lambda)y)\leq\lambda f(x)+(1-\lambda)f(y)-\frac{\sigma}{2}\lambda(1-\lambda)\|x-y\|^2.$$
\end{defn}
\begin{defn}
A function $f\in\Gamma_0(\R^n)$ is \emph{strictly convex} if for all $x,y\in\dom f,$ $x\neq y$ and all $\lambda\in(0,1),$
$$f(\lambda x+(1-\lambda)y)<\lambda f(x)+(1-\lambda)f(y).$$
\end{defn}
\begin{defn}
A function $f\in\Gamma_0(\R^n)$ is \emph{essentially strictly convex} if $f$ is strictly convex on every convex subset of $\dom\partial f.$
\end{defn}
\noindent Next, we have some facts about the Moreau envelope, including differentiability, upper and lower bounds, pointwise convergence characterization, linear translation and evenness.
\begin{Fact}[Inverse Function Theorem]\emph{\cite[Theorem 5.2.3]{hamilton1982inverse}}\label{fact:inverse}
Let $f:U\to\R^n$ be $\mathcal{C}^k$ on the open set $U\subseteq\R^n.$ If at some point the Jacobian of $f$ is invertible, then there exist $V\subseteq\R^n$ open and $g:V\to\R^n$ of class $\mathcal{C}^k$ such that
\begin{itemize}
\item[(i)]$v_0=f(u_0)\in V$ and $g(v_0)=u_0;$
\item[(ii)]$U_0=g(V)$ is open and contained in $U;$
\item[(iii)]$f(g(v))=v~\forall v\in V.$
\end{itemize}
Thus, $f:U_0\to V$ is a bijection and has inverse $g:V\to U_0$ of class $\mathcal{C}^k.$
\end{Fact}

\begin{Fact}\emph{\cite[Proposition 12.9]{convmono}, \cite[Theorem 1.25]{rockwets}}\label{fact:egoestoinff}
Let $f\in\Gamma_0(\R^n).$ Then for all $x\in\R^n,$
\begin{itemize}
\item[(i)] $\inf f\leq e_rf(x)\leq f(x),$
\item[(ii)] $\lim\limits_{r\nearrow\infty}e_rf(x)=f(x),$ and
\item[(iii)] $\lim\limits_{r\searrow0}e_rf(x)=\inf f.$
\end{itemize}
\end{Fact}

\begin{Fact}\emph{\cite[Theorem 7.37]{rockwets}}\label{fact:epi}
Let $\{f^\nu\}_{\nu\in\N}\subseteq\Gamma_0(\R^n)$ and $f\in\Gamma_{0}(\R^n).$ Then $f^\nu\overset{e}\rightarrow f$ if and only if $e_rf^\nu\overset{p}\rightarrow e_rf$.
Moreover, the pointwise convergence of $e_rf^\nu$ to $e_rf$ is uniform on all bounded subsets of $\R^n,$ hence yields epi-convergence to $e_rf$ as well.
\end{Fact}

\begin{Fact}\emph{\cite[Lemma 2.2]{proxmap}}\label{affineshift}
Let $f:\R^n\rightarrow\overline{\R}$ be proper lsc, and $g(x)=f(x)-a^\top x$ for some $a\in\R^n.$ Then
$$e_rg(x)=e_rf\left(x+\frac{a}{r}\right)-a^\top x-\frac{1}{2r}a^\top a.$$
\end{Fact}

\begin{lem}\label{evenlem}
Let $f:\R^n\rightarrow\ORR$ be an even function. Then $e_rf$ is an even function.
\end{lem}
\begin{proof} Let $f(-x)=f(x).$ Then
\begin{align*}
(e_rf)(-x)&=\inf\limits_{z\in\R^n}\left\{f(z)+\frac{r}{2}\|z-(-x)\|^2\right\}.
\end{align*}
Let $z=-y.$ Then
\begin{align*}
(e_rf)(-x)&=\inf\limits_{-y\in\R^n}\left\{f(-y)+\frac{r}{2}\|-y-(-x)\|^2\right\}\\
&=\inf\limits_{-y\in\R^n}\left\{f(y)+\frac{r}{2}\|y-x\|^2\right\}\\
&=\inf\limits_{y\in\R^n}\left\{f(y)+\frac{r}{2}\|y-x\|^2\right\}\\
&=(e_rf)(x).\qedhere
\end{align*}
\end{proof}

\section{Properties of the Moreau envelope of convex functions}\label{sec:props}

In this section, we present results on bounds and differentiability, and follow up with characterizations that involve strict convexity, strong convexity and Lipschitz continuity. These results are the setup for the two sections that follow, where we explore more specific families of functions.

\subsection{The set of convex Moreau envelopes}

We begin by providing several properties of Moreau envelopes of proper, lsc, convex functions. We show that the set of all such envelopes is closed and convex, and we give a bound for $f-e_rf$ when $f$ is Lipschitz continuous. The facts in this section are already known in the literature, but they are scattered among several articles and books, so it is convenient to have them all in one collection.
\begin{Fact}\emph{\cite[Theorem 3.18]{wangklee}} Let $f:\R^n\rightarrow\ORR$ be proper and lsc. Then $e_rf=f$ for some $r>0$ if and only if $f$ is a constant function.
\end{Fact}
\begin{Fact}\emph{\cite[Theorem 3.1]{planwang2016}}\label{prop:1}
The set $e_r(\Gamma_0(\R^n))$ is a convex set in $\Gamma_0(\R^n).$
\end{Fact}
%
\begin{Fact}\emph{\cite[Theorem 3.2]{planwang2016}}
The set $e_r(\Gamma_0(\R^n))$ is closed under pointwise convergence.
\end{Fact}
\begin{prop}\label{thm:lipbound}
Let $f\in\Gamma_0(\R^n)$ be $L$-Lipschitz. Then for all $x\in\dom f$ and any $r>0,$$$0\leq f(x)-e_rf(x)\leq\frac{L^2}{2r}.$$
\end{prop}
\begin{proof}
We have $0\leq f(x)-e_rf(x)~\forall x\in\dom f$ by \cite[Theorem 1.25]{rockwets}. Then
\begin{align*}
f(x)-e_rf(x)&=f(x)-\left[f(P_rf(x))+\frac{r}{2}\|x-P_rf(x)\|^2\right]\\
&\leq L\|x-P_rf(x)\|-\frac{r}{2}\|x-P_rf(x)\|^2\\
&=Lt-\frac{r}{2}t^2,
\end{align*}where $t=\|x-P_rf(x)\|.$ This is a concave quadratic function whose maximizer is $L/r.$ Thus,
\begin{align*}
Lt-\frac{r}{2}t^2&\leq L\frac{L}{r}-\frac{r}{2}\left(\frac{L}{r}\right)^2\\
&=\frac{L^2}{2r}.\qedhere
\end{align*}
\end{proof}
\noindent The following example demonstrates that for an affine function, the bound in Proposition \ref{thm:lipbound} is tight.
\begin{ex}\label{ex:affinetight}
Let $f:\R^n\to\ORR,$ $f(x)=\langle a,x\rangle+b,$ $a\in\R^n,$ $b\in\R.$ Then $$f-e_rf=\frac{\|a\|^2}{2r}.$$
\end{ex}
\begin{proof}
We have
$$e_rf(x)=\inf\limits_{y\in\R^n}\left\{\langle a,y\rangle+b+\frac{r}{2}\|y-x\|^2\right\}=\inf\limits_{y\in\R^n}g(y).$$
Setting $g'(y)=0$ to find critical points yields $y=x-a/r.$
Substituting into $g(y),$ we have
$$e_rf(x)=\left\langle a,x-\frac{a}{r}\right\rangle+\frac{r}{2}\left(x-\frac{a}{r}-x\right)^2=\langle a,x\rangle+b-\frac{\|a\|^2}{2r}.$$
Thus,$$f-e_rf=\frac{\|a\|^2}{2r},$$where $\|a\|$ is the Lipschitz constant of the affine function $f.$
\end{proof}
\noindent The next theorem is a characterization of when a convex function and its Moreau envelope differ only by a constant: when the function is affine.
\begin{thm}
Let $f\in\Gamma_0(\R^n),$ $r>0.$ Then $f=e_rf+c$ for some $c\in\R$ if and only if $f$ is an affine function.
\end{thm}
\begin{proof}
$(\Leftarrow)$ This is the result of Example \ref{ex:affinetight}.

\noindent$(\Rightarrow)$ Suppose that $f=e_rf+c.$ Taking the Fenchel conjugate of both sides and rearranging, we have
\begin{equation}\label{eq:affineenvelope}
f^*+\frac{1}{2r}\|\cdot\|^2=f^*+c.
\end{equation}
Let $x_0$ be such that $f^*(x_0)<\infty.$ Then by \eqref{eq:affineenvelope} we have
\begin{align}
f^*(x_0)+\frac{1}{2r}\|x_0\|^2&=f^*(x_0)+c\nonumber\\
c&=\frac{1}{2r}\|x_0\|^2.\label{eq:cx0}
\end{align}
Now suppose there exists $x_1\neq x_0$ such that $f^*(x_1)<\infty.$ Then $\co\{x_0,x_1\}\subseteq\dom f^*,$ thus, $\frac{1}{2r}\|tx_0+(1-t)x_1\|^2=c$ for all $t\in[0,1].$ Substituting \eqref{eq:cx0}, we have
\begin{align}
t^2\|x_0\|^2+2t(1-t)\langle x_0,x_1\rangle+(1-t)^2\|x_1\|^2&=\|x_0\|^2\nonumber\\
(t^2-1)\|x_0\|^2-2t(t-1)\langle x_0,x_1\rangle+(t-1)^2\|x_1\|^2&=0\nonumber\\
(t+1)\|x_0\|^2-2t\langle x_0,x_1\rangle+(t-1)\|x_1\|^2&=0.\label{notequal0}
\end{align}
Note that if one of $x_0,x_1$ equals zero, then \eqref{notequal0} implies that the other one equals zero, a contradiction to $x_0\neq x_1.$ Hence, $x_0\neq0$ and $x_1\neq0.$ Since the left-hand side of \eqref{notequal0} is a smooth function of $t$ for $t\in(0,1),$ we take the derivative of \eqref{notequal0} with respect to $t$ and obtain
\begin{align*}
\|x_0\|^2-2\langle x_0,x_1\rangle+\|x_1\|^2&=0\\
\|x_0-x_1\|^2&=0\\
x_0&=x_1,
\end{align*}
a contradiction. Hence, $\dom f^*=\{x_0\}.$ Therefore, $f^*=\iota_{\{x_0\}}+k$ for some $k\in\R,$ and we have $f(x)=\langle x,x_0\rangle-k.$
\end{proof}

\subsection{Characterizations of the Moreau envelope}

Now we  show the ways in which $e_rf$ can be characterized in terms of $f$ when $f$ has a certain structure. We consider the properties of strict convexity, strong convexity, and Lipschitz continuity.
\begin{thm}\label{strictcon}
Let $f\in\Gamma_0(\R^n).$ Then $f$ is essentially strictly convex if and only if $e_rf$ is strictly convex.
\end{thm}
\begin{proof} Let $f$ be essentially strictly convex. By \cite[Theorem 26.3]{convanalrock}, we have that $f^*$ is essentially smooth, as is $(e_rf)^*=f^*+\frac{1}{2r}\|\cdot\|^2$. Applying \cite[Theorem 26.3]{convanalrock} gives us that $e_rf$ is essentially strictly convex. Since Moreau envelopes of convex functions are convex and full-domain, this essentially strict convexity is equivalent to strict convexity. Therefore, $e_rf$ is strictly convex. Conversely, assuming that $e_rf$ is strictly convex, the previous statements in reverse order allow us to conclude that $f$ is essentially strictly convex.
\end{proof}
\begin{thm}\label{strongcon}
Let $g\in\Gamma_0(\R^n).$ Then $f=e_rg$ if and only if $f^*$ is strongly convex with modulus $1/r.$
\end{thm}
\begin{proof}
$(\Rightarrow)$ Suppose $f=e_rg.$ Making use of the Fenchel conjugate, we have
\begin{align*}
f&=e_rg\\
f^*&=g^*+\frac{1}{2r}\|\cdot\|^2\\
g^*&=f^*-\frac{1}{2r}\|\cdot\|^2.
\end{align*}
Since the conjugate of $g\in\Gamma_0(\R^n)$ is again a function in $\Gamma_0(\R^n),$ we have that $f^*-\frac{1}{2r}\|\cdot\|^2$ is in $\Gamma_0(\R^n),$ which means that $f^*$ is strongly convex with modulus $1/r.$\medskip

\noindent$(\Leftarrow)$ Suppose $f^*$ is strongly convex with modulus $1/r.$ Then $f^*-\frac{1}{2r}\|\cdot\|^2=g^*$ for some $g^*$ in $\Gamma_0(\R^n),$ and we have
\begin{align*}
f^*&=g^*+\frac{1}{2r}\|\cdot\|^2\\
&=g^*+\left(\frac{r}{2}\|\cdot\|^2\right)^*.
\end{align*}
Taking the Fenchel conjugate of both sides, and invoking \cite[Theorem 16.4]{convmono}, we have (using \framebox(5,5) ~~as the infimal convolution operator)
\begin{align*}
f&=\left[g^*+\left(\frac{r}{2}\|\cdot\|^2\right)^*\right]^*\\
&=g^{**}~\framebox(5,5)~~\frac{r}{2}\|\cdot\|^2\\
&=g~\framebox(5,5)~~\frac{r}{2}\|\cdot\|^2\\
&=e_rg.\qedhere
\end{align*}
\end{proof}
\begin{Fact}\label{1lip}\emph{\cite[Corollary 18.18]{convmono}}
Let $g\in\Gamma_0(\R^n).$ Then $g=e_rf$ for some $f\in\Gamma_0(\R^n)$ if and only if $\nabla g$ is r-Lipschitz.
\end{Fact}
\begin{Fact}\label{lemstrcon}\emph{\cite[Lemma 2.3]{planwang2016}}
Let $r>0$. The function $f\in\Gamma_0(\R^n)$ is $r$-strongly convex if anf only if $e_1f$ is $\frac{r}{r+1}$-strongly convex.
\end{Fact}
\noindent For strongly convex functions, a result that resembles the combination of Facts \ref{1lip} and \ref{lemstrcon} is found in \cite{practasp}. This is a reciprocal result, in that it is not $e_rf$ that is found to have a Lipschitz gradient as in Fact \ref{1lip}, but $(e_1f)^*$.\footnote{Thank you to the anonymous referee for providing this reference.}
\begin{Fact}\emph{\cite[Theorem 2.2]{practasp}} Let $f$ be a finite-valued convex function. For a symmetric linear operator $M\in S^n_{++},$ define $\langle\cdot,\cdot\rangle_M=\langle M\cdot,\cdot\rangle,$ $\|\cdot\|^2_M=\langle \cdot,\cdot\rangle_M$ and
$$F(x)=\inf\limits_{y\in\R^n}\left\{f(y)+\frac{1}{2}\|y-x\|^2_M\right\}.$$ Then the following are equivalent:
\begin{itemize}
\item[(i)] $f$ is $\frac{1}{k}$-strongly convex;
\item[(ii)] $\nabla f^*$ is $k$-Lipschitz;
\item[(iii)] $\nabla F^*$ is $K$-Lipschitz;
\item[(iv)] $F$ is $\frac{1}{K}$-strongly convex;
\end{itemize}
for some $K$ such that $k-1/\lambda\leq K\leq k+1/\lambda,$ where $\lambda$ is the minimum eigenvalue of $M.$
\end{Fact}

\subsection{Differentiability of the Moreau envelope}

It is well known that $e_{r}f$ is differentiable if $f\in\Gamma_{0}(\R^n)$; see \cite{rockwets}. In this section, we study differentiability of $e_rf$ when $f$ enjoys higher-order differentiability.

\begin{thm} Let $f\in \Gamma_0(\R^n)$ and $f\in\mathcal{C}^k.$ Then $e_rf\in\mathcal{C}^k.$
\end{thm}
\begin{proof}
If $k=1,$ the proof is that of \cite[Proposition 13.37]{rockwets}. Assume $k>1.$ Since $f\in\Gamma_0(\R^n),$ by \cite[Theorem 2.26]{rockwets} we have that \begin{equation}\label{eq:erfid}\nabla e_rf=r\Id-r\left(\Id+\frac{1}{r}\nabla f\right)^{-1},\end{equation} and that $P_rf=\left(\Id+\frac{1}{r}\nabla f\right)^{-1}$ is unique for each $x\in\dom f.$ Let $y=\left(\Id+\frac{1}{r}\nabla f\right)^{-1}(x).$ Then $x=y+\frac{1}{r}\nabla f(y)=:g(y),$ and for any $y_0\in\dom f$ we have
$$\nabla g(y_0)=\Id+\frac{1}{r}\nabla^2f(y_0),$$
where $\nabla^2f(y_0)\in\R^{n\times n}$ exists (since $f\in\mathcal{C}^2$) and is positive semidefinite. This gives us that $\nabla g\in\mathcal{C}^{k-2},$ so that $g\in\mathcal{C}^{k-1}.$ Then by Fact \ref{fact:inverse}, we have that $g^{-1}=P_rf\in\mathcal{C}^{k-1}.$ Thus, by \eqref{eq:erfid} we have that $\nabla e_rf\in\mathcal{C}^{k-1}.$ Therefore, $e_rf\in\mathcal{C}^k.$
\end{proof}
\subsection{Moreau envelopes of piecewise differentiable functions}
When a function is piecewise differentiable, using Minty's surjective theorem, we can provide a closed analytical form for its Moreau
envelope. This section is the setup for the main result of Section \ref{sec:pcf}, in which Theorem \ref{thm:pcf} gives the explicit
expression of the Moreau envelope for a piecewise cubic function on $\R.$

\begin{prop}\label{prop:pieces}
Let $f_1,f_2:\R\rightarrow \R$ be convex and differentiable on the whole of $\R$ such that
$$f(x)=\begin{cases}
f_1(x),&\mbox{\emph{if} }x\leq x_0\\
f_2(x),&\mbox{\emph{if} }x\geq x_0
\end{cases}$$
is convex. Then
\begin{align*}
P_rf(x)&=\begin{cases}
P_rf_1(x),&\mbox{\emph{if} }x<x_0+\frac{1}{r}f_1'(x_0),\\
x_0,&\mbox{\emph{if} }x_0+\frac{1}{r}f_1'(x)\leq x\leq x_0+\frac{1}{r}f_2'(x_0),\\
P_rf_2(x),&\mbox{\emph{if} }x>x_0+\frac{1}{r}f_2'(x_0),
\end{cases}\\
e_rf(x)&=\begin{cases}
e_rf_1(x),&\mbox{\emph{if} }x<x_0+\frac{1}{r}f_1'(x_0),\\
f_1(x_0)+\frac{r}{2}(x_0-x)^2,&\mbox{\emph{if} }x_0+\frac{1}{r}f_1'(x_0)\leq x\leq x_0+\frac{1}{r}f_2'(x_0),\\
e_rf_2(x),&\mbox{\emph{if} }x>x_0+\frac{1}{r}f_2'(x_0).
\end{cases}
\end{align*}
\end{prop}
\begin{proof}
First observe that since $f$ is convex, $f_1(x_0)=f_2(x_0)$ and $f_1'(x_0)\leq f_2'(x_0)$. Hence, $f$ is continuous, and the regions $x<x_0+\frac{1}{r}f_1'(x_0)$ and $x>x_0+\frac{1}{r}f_2'(x_0)$ cannot overlap. We split the Moreau envelope as follows,
\begin{equation}\label{eq:twoinfima}
e_rf(x)=\min\left[\inf\limits_{y<x_0}\left\{f_1(y)+\frac{r}{2}(y-x)^2\right\},\inf\limits_{y\geq x_0}\left\{f_2(y)+\frac{r}{2}(y-x)^2\right\}\right].
\end{equation}
\noindent {\sl Case 1: $x<x_0+\frac{1}{r}f_1'(x_0)$.} We show that $$e_rf(x)=e_rf_1(x)=\inf\limits_{y<x_0}\left\{f_1(y)+\frac{r}{2}(y-x)^2\right\},\mbox{ and }$$
$$P_rf(x)=P_rf_{1}(x)<x_0.$$
On $(-\infty, x_0)$ the function $y\mapsto f(y)+\frac{r}{2}(y-x)^2=f_1(y)+\frac{r}{2}(y-x)^2$
is convex, so any local minimizer will be a global minimizer for the function
$y\mapsto f(y)+\frac{r}{2}(y-x)^2$ and  $y\mapsto f_1(y)+\frac{r}{2}(y-x)^2$ on $\R$, which in turn
imply that $P_{r}f(x)=P_{r}f_{1}(x)$.
It suffices to show that
\begin{equation}\label{eq:interior}
x<x_0+\frac{1}{r}f_1'(x_0)\Rightarrow\argmin_{y<x_0}\left\{f_1(y)+\frac{r}{2}(y-x)^2\right\}
\in (-\infty,x_0).
\end{equation}
The existence of the minimizer is guaranteed by the convexity of $f_{1}$, which
implies the coercivity of $y\mapsto f_1(y)+\frac{r}{2}(y-x)^2$.
Then we will have
\begin{equation}\label{eq:proxresult1}
\argmin_{y<x_0}\left\{f_1(y)+\frac{r}{2}(y-x)^2\right\}=P_rf_1(x)=P_{r}f(x) \quad \forall x<x_0+\frac{1}{r}f'_{1}(x_{0}).
\end{equation}
We show \eqref{eq:interior} by contradiction. Assume that $\argmin_{y<x_0}\{f_1(y)+\frac{r}{2}(y-x)^2\}=\{x_0\}$ under the condition $x<x_0+\frac{1}{r}f_1'(x_0)$. Then
$$f_1(x_0)+\frac{r}{2}(x_0-x)^2\leq f_1(y)+\frac{r}{2}(y-x)^2,~\forall y\leq x_0.$$ Hence, $f_1(y)+\frac{r}{2}(y-x)^2+\iota_{(-\infty,x_0]}$ attains a global minimum at $y=x_0$. By the optimality condition,
\begin{align*}
0&\in f_1'(x_0)+r(x_0-x)+\R_+,\\
0&\in\frac{1}{r}f_1'(x_0)+x_0-x+\R_+,\\
x&\in\frac{1}{r}f_1'(x_0)+x_0+\R_+.
\end{align*}
So $x=x_0+\frac{1}{r}f_1'(x_0)+t$ for some $t\geq0,$ which is a contradiction. Thus, \eqref{eq:interior} holds, and we conclude \eqref{eq:proxresult1}.

\noindent {\sl Case 2: $x>x_0+\frac{1}{r}f_2'(x_0)$.} We show that $$e_rf(x)=e_rf_2(x)=\inf\limits_{y>x_0}\left\{f_2(y)+\frac{r}{2}(y-x)^2\right\},\mbox{ and }$$
$$P_rf(x)=P_rf_{2}(x)>x_0.$$
This is realized by an identical argument as in Case 1.\\
\noindent {\sl Case 3: $x_0+\frac{1}{r}f_1'(x_0)\leq x\leq x_0+\frac{1}{r}f_2'(x_0).$}

In this region, we must have $P_rf(x)=x_0$.  Indeed, since
$$f_1'(x_0)+r(x_{0}-x)\leq 0 \leq f_2'(x_0) +r(x_{0}-x), \text{ and }$$
$$\partial f(x_{0})=[f'_{1}(x_{0}),f_{2}'(x_{0})],$$
 we have
$$0\in [f_1'(x_0)+r(x_{0}-x), f_2'(x_0) +r(x_{0}-x)]=\partial \left(f+\frac{r}{2}(\cdot-x)^2\right)(x_{0}).$$
By convexity, this means that $y\mapsto f(y)+\frac{r}{2}(y-x)^2$ attains its global minimum at $x_{0}$,
and that $P_{r}f(x)=x_{0}$.
Because both infima in \eqref{eq:twoinfima} yield the same expression. We use the first one without loss of generality and conclude the remainder of the statement of the proposition.
\end{proof}
\noindent Proposition \ref{prop:pieces} can be expanded to any finite number of functions with the same manner of proof.
\begin{cor}\label{pcubthm}
Let $x_1<\cdots<x_n.$ Let $f_0,f_1,\ldots,f_m:\R\rightarrow\overline{\R}$ be differentiable on the whole of $\R$ such that
$$f(x)=\begin{cases}
f_0(x),&\mbox{\emph{if} }x\leq x_1,\\
f_1(x),&\mbox{\emph{if} }x_1\leq x\leq x_2,\\
&\vdots\\
f_m(x),&\mbox{\emph{if} }x_m\leq x.
\end{cases}$$
is convex. Then
$$P_rf(x)=\begin{cases}
P_rf_0(x),&\mbox{\emph{if} }x<x_1+\frac{1}{r}f_0'(x_1),\\
x_1,&\mbox{\emph{if} }x_1+\frac{1}{r}f_0'(x_1)\leq x\leq x_1+\frac{1}{r}f_1'(x_1),\\
P_rf_1(x),&\mbox{\emph{if} }x_1+\frac{1}{r}f_1'(x_1)<x<x_2+\frac{1}{r}f_1'(x_2),\\
&\vdots\\
x_m,&\mbox{\emph{if} }x_m+\frac{1}{r}f_{m-1}'(x_m)\leq x\leq x_m+\frac{1}{r}f_m'(x_m),\\
P_rf_m(x),&\mbox{\emph{if} }x_m+\frac{1}{r}f_m'(x_m)<x,
\end{cases}$$
$$e_rf(x)=\begin{cases}
e_rf_0(x),&\mbox{\emph{if} }x<x_1+\frac{1}{r}f_0'(x_1),\\
f_1(x_1)+\frac{r}{2}(x_1-x)^2,&\mbox{\emph{if} }x_1+\frac{1}{r}f_0'(x_1)\leq x\leq x_1+\frac{1}{r}f_1'(x_1),\\
e_rf_1(x),&\mbox{\emph{if} }x_1+\frac{1}{r}f_1'(x_1)<x<x_2+\frac{1}{r}f_1'(x_2),\\
&\vdots\\
f_m(x_m)+\frac{r}{2}(x_m-x)^2,&\mbox{\emph{if} }x_m+\frac{1}{r}f_{m-1}'(x_m)\leq x\leq x_m+\frac{1}{r}f_m'(x_m),\\
e_rf_m(x),&\mbox{\emph{if} }x_m+\frac{1}{r}f_m'(x_m)<x.
\end{cases}$$
\end{cor}
\begin{proof}
In the definition of $f,$ we have that $f_{i-1}(x_i)=f_i(x_i),$ so that $f$ is continuous. Since $f$ is convex, $P_rf(x)$ is monotone. We split the Moreau envelope as follows,
\begin{align*}
e_rf(x)=\min\bigg[&\inf\limits_{y\leq x_1}\left\{f_0(y)+\frac{r}{2}(y-x)^2\right\},\\
&\inf\limits_{x_1\leq y\leq x_2}\left\{f_1(y)+\frac{r}{2}(y-x)^2\right\},\\
&\vdots\\
&\inf\limits_{x_m\leq y}\left\{f_m(y)+\frac{r}{2}(y-x)^2\right\}\bigg],
\end{align*}
The rest of the proof is identical in method to that of Proposition \ref{prop:pieces}.\end{proof}
\noindent The following is an example of Corollary \ref{pcubthm}, with a three-piece function.
\begin{ex}
Define $f:\R\to\R,$$$f(x)=\begin{cases}
-5x-2,&\mbox{if }x<-1,\\
(x-1)^2-1,&\mbox{if }-1\leq x\leq0,\\
x^3,&\mbox{if }x>0.
\end{cases}$$Then
$$P_rf(x)=\begin{cases}
x+\frac{5}{r},&\mbox{if }x<-1-\frac{5}{r},\\
-1,&\mbox{if }-1-\frac{5}{r}\leq x\leq-1-\frac{4}{r},\\
\frac{rx+2}{r+2},&\mbox{if }-1-\frac{4}{r}<x<-\frac{2}{r},\\
0,&\mbox{if }-\frac{2}{r}\leq x\leq0,\\
\frac{-r+\sqrt{r^2+12rx}}{6},&\mbox{if }x>0
\end{cases}$$and$$e_rf(x)=\begin{cases}
-5x-\frac{25}{2r}-2,&\mbox{if }x<-1-\frac{5}{r},\\
\frac{r}{2}(x+1)^2+3,&\mbox{if }-1-\frac{5}{r}\leq x\leq-1-\frac{4}{r},\\
\frac{r}{r+2}(x-1)^2-1,&\mbox{if }-1-\frac{4}{r}<x<-\frac{2}{r},\\
\frac{r}{2}x^2,&\mbox{if }-\frac{2}{r}\leq x\leq0,\\
\frac{r^3-r(r+12x)\sqrt{r^2+12rx}+18r^2x+54rx^2}{108},&\mbox{if }x>0.\end{cases}$$
\end{ex}
\begin{proof}
The proof is a matter of applying Corollary \ref{pcubthm} with $x_1=-1,$ $x_2=0,$ $f_0(x)=-5x-2,$ $f_1(x)=(x-1)^2-1$ and $f_2(x)=x^3.$ The algebra and calculus are elementary and are left to the reader as an exercise.
\end{proof}
\noindent Figure \ref{fig:threepieceex} presents $f$ and $e_rf$ for several values of $r.$
\begin{figure}[H]\begin{center}
\includegraphics[scale=0.3]{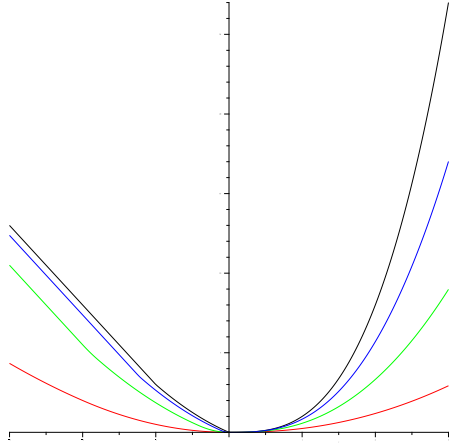}\end{center}
\caption{The functions $f$ (black) and $e_rf$ for $r=1$ (red), $5$ (green) and $20$ (blue).}\label{fig:threepieceex}
\end{figure}
\begin{thm}\label{thm:inf}
Let $f\in\Gamma_0(\R)$ be differentiable on $[a,b].$ Define $g\in\Gamma_0(\R)$ by
$$g(x)=\begin{cases}
f(x),&\mbox{\emph{if} }a\leq x\leq b,\\\infty,&\mbox{\emph{otherwise.}}
\end{cases}$$
Then
$$P_rg(x)=\begin{cases}
a,&\mbox{\emph{if} }x\leq a+\frac{1}{r}f'(a),\\
P_rf(x),&\mbox{\emph{if} }a+\frac{1}{r}f'(a)<x<b+\frac{1}{r}f'(b),\\
b,&\mbox{\emph{if} }b+\frac{1}{r}f'(b)\leq x,
\end{cases}$$ and
$$e_rg(x)=\begin{cases}
f(a)+\frac{r}{2}(a-x)^2,&\mbox{\emph{if} }x\leq a+\frac{1}{r}f'(a),\\
e_rf(x),&\mbox{\emph{if} }a+\frac{1}{r}f'(a)<x<b+\frac{1}{r}f'(b),\\
f(b)+\frac{r}{2}(b-x)^2,&\mbox{\emph{if} }b+\frac{1}{r}f'(b)\leq x.
\end{cases}$$
\end{thm}
\begin{proof} We use the fact that $P_rg=\left(\Id+\frac{1}{r}\partial g\right)^{-1}$ \cite[Example 23.3]{convmono}. Find the subdifferential of $g:$
$$\partial g(x)=\begin{cases}
f'(x),&\mbox{if }a<x<b,\\
f'(a)+\R_-,&\mbox{if }x=a,\\
f'(b)+\R_+,&\mbox{if }x=b,\\
\varnothing,&\mbox{otherwise.}
\end{cases}$$
Multiplying by $\frac{1}{r}$ and adding the identity function, we obtain
$$x+\frac{1}{r}\partial g(x)=\begin{cases}
x+\frac{1}{r}f'(x),&\mbox{if }a<x<b,\\
a+\frac{1}{r}f'(a)+\R_-,&\mbox{if }x=a,\\
b+\frac{1}{r}f'(b)+\R_+,&\mbox{if }x=b,\\
\varnothing,&\mbox{otherwise.}
\end{cases}$$
Now applying the identity $P_rg(x)=\left(\Id+\frac{1}{r}\partial g\right)^{-1}(x),$ we find
\begin{equation*}
P_rg(x)=\begin{cases}
P_rf(x),&\mbox{if }a+\frac{1}{r}f'(a)<x<b+\frac{1}{r}f'(b),\\
a,&\mbox{if }a+\frac{1}{r}f'(a)\geq x,\\
b,&\mbox{if }b+\frac{1}{r}f'(b)\leq x.
\end{cases}\end{equation*}
\end{proof}
\begin{ex}
Define
$$g(x)=\begin{cases}x,&\mbox{if }-1\leq x\leq2,\\\infty,&\mbox{otherwise.}\end{cases}$$Then by Theorem \ref{thm:inf},
$$e_rg(x)=\begin{cases}-1+\frac{r}{2}(-1-x)^2,&\mbox{if }x\leq-1+\frac{1}{r},\\
x-\frac{1}{2r},&\mbox{if }-1+\frac{1}{2}<x\leq2+\frac{1}{r},\\
2+\frac{r}{2}(2-x)^2,&\mbox{if }x>2+\frac{1}{r}.\end{cases}$$
\end{ex}
\begin{figure}[H]\begin{center}
\includegraphics[scale=0.2]{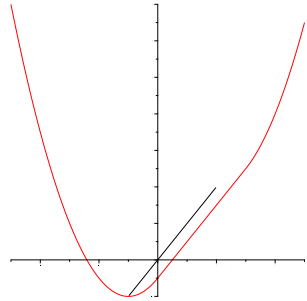}\end{center}
\caption{The functions $g$ (black) and $e_1g$ (red).}
\end{figure}

\section{The Moreau envelope of piecewise cubic functions}\label{sec:pcf}

In this section, we concentrate our efforts on the class of univariate, piecewise cubic functions.

\subsection{Motivation}

Piecewise polynomial functions are of great interest in current research because they are commonly used in mathematical modelling, and thus in many optimization algorithms that require a relatively simple approximation function. Convex piecewise functions in general, and their Moreau envelopes, are explored in \cite{meng2001piecewise,mifflin1999properties} and similar works. Properties of piecewise linear-quadratic (PLQ) functions in particular, and their Moreau envelopes, are developed in \cite{bajaj2017visualization,convhullalg,plqmodel} and others. The new theory of piecewise cubic functions found in this section will enable the expansion of such works to polynomials of one degree higher, and any result developed here reverts to the piecewise linear-quadratic case by setting the cubic coefficients to zero. Matters such as interpolation for discrete transforms, closedness under Moreau envelope, and efficiency of Moreau envelope algorithms that are analyzed in \cite{plqmodel} for PLQ functions can now be extended to the piecewise cubic case, as can the PLQ Toolbox software found in \cite[\S7]{plqmodel}. Indeed, it is our intention that many applications and algorithms that currently use PLQ functions as their basis will become applicable to a broader range of useful situations due to expansion to the piecewise-cubic setting.

\subsection{Convexity}

We begin with the definition and a lemma that characterizes when a piecewise cubic function is convex.
\begin{defn}\label{pcub}
A function $f:\R\rightarrow\R$ is called \emph{piecewise cubic} if $\dom f$ can be represented as the union of finitely many closed intervals, relative to each of which $f(x)$ is given by an expression of the form $ax^3+bx^2+cx+d$ with $a,b,c,d\in\R.$
\end{defn}
\begin{prop}\label{propcont}
If a function $f:\R\rightarrow\R$ is piecewise cubic, then $\dom f$ is closed and $f$ is continuous relative to $\dom f,$ hence lsc on $\R.$
\end{prop}
\begin{proof}
The proof is the same as that of \cite[Proposition 10.21]{rockwets}.
\end{proof}
\begin{lem}\label{lem:piecewisecubic}
For $i=1,2,\ldots,m,$ let $f_i$ be a cubic, full-domain function on $\R,$
$$f_i(x)=a_ix^3+b_ix^2+c_ix+d_i.$$
For $i=1,2,\ldots,m-1,$ let $\{x_i\}$ be in increasing order, $x_1<x_2<\cdots<x_{m-1},$ such that
$$f_i(x_i)=f_{i+1}(x_i).$$
Define the subdomains
$$D_1=(-\infty,x_1], D_2=[x_1,x_2], \ldots, D_{m-1}=[x_{m-2},x_{m-1}], D_m=[x_{m-1},\infty).$$
Then the function $f$ defined by
$$f(x)=\begin{cases}
f_1(x),&\mbox{\emph{if} }x\in D_1,\\
f_2(x),&\mbox{\emph{if} }x\in D_2,\\
&\vdots\\
f_{m-1}(x),&\mbox{\emph{if} }x\in D_{m-1},\\
f_m(x),&\mbox{\emph{if} }x\in D_m
\end{cases}$$
is a continuous, piecewise cubic function. Moreover, $f$ is convex if and only if
\begin{itemize}
\item[(i)] $f_i$ is convex on $D_i$ for each $i,$ and
\item[(ii)] $f_i'(x_i)\leq f_{i+1}'(x_i)$ for each $i<m.$
\end{itemize}
\end{lem}
\begin{proof} By Proposition \ref{propcont}, $f$ is a continuous, piecewise cubic function.\medskip\\$(\Leftarrow)$  Suppose that each $f_i$ is convex on $D_i$ and that $f_i'(x_i)\leq f_{i+1}'(x_i)$ for each $i<m.$ Since $f_i$ is convex and smooth on $\intt D_i$ for each $i$, we have that for each $i:$
\begin{itemize}
\item[(a)] $f_i'$ is monotone on $\intt D_i,$
\item[(b)] $f_i'(x_i)=\sup\limits_{x\in\intt D_i}f_i'(x)$ (by point (a), and because $f_i$ is polynomial $f_i'$ is continuous, $f'_i$ is an increasing function), and
\item[(c)] $f_{i+1}'(x_i)=\inf\limits_{x\in\intt D_{i+1}}f'_{i+1}(x)$ (by point (a) and continuity of $f_i'$).
\end{itemize}
Then at each $x_i,$ the subdifferential of $f$ is the convex hull of $f_i'(x_i)$ and $f_{i+1}'(x_i):$
\begin{itemize}
\item[(d)] $\partial f(x_i)=[f_i'(x_i),f_{i+1}'(x_i)].$
\end{itemize}
Points (a), (b), (c), and (d) above give us that $\partial f$ is monotone over its domain. Therefore, $f$ is convex.\medskip\\
$(\Rightarrow)$ Suppose that $f$ is convex. It is clear that if $f_i$ is not convex on $D_i$ for some $i,$ then $f$ is not convex and we have a contradiction. Hence, $f_i$ is convex on $D_i$ for each $i,$ and point (i) is true. Suppose for eventual contradiction that $f'_{i+1}(x_i)<f'_i(x_i)$ for some $i<m.$ Since point (i) is true, point (a) and hence point (b) are also true. Thus, since $f'_i$ is a continuous function on $D_i,$ there exists $x\in\intt D_i$ such that $f_i'(x)>f_{i+1}'(x_i).$ Since $x<x_i,$ we have that $\partial f$ is not monotone. Hence, $f$ is not convex, a contradiction. Therefore, $f_i'(x_i)\leq f_{i+1}'(x_i)$ for all $i<m.$
\end{proof}

\subsection{Examples}

It will be helpful to see how the Moreau envelopes of certain piecewise cubic functions behave graphically. Visualizing a few simple functions and their Moreau envelopes points the way to the main results in the next section.
\begin{ex}
Let $x_1=-1,$ $x_2=1.$ Define
$$f_0(x)=-2x^3+2x^2+2x+3,~f_1(x)=x^3+3x^2-x+2,~f_2(x)=3x^3+2x^2+2x-2,$$
$$f(x)=\begin{cases}
f_0(x),&\mbox{if }x<x_1\\
f_1(x),&\mbox{if }x_1\leq x<x_2,\\
f_2(x),&\mbox{if }x_2\leq x.
\end{cases}$$
\end{ex}
\noindent It is left to the reader to verify that $f$ is convex. Notice that $x_1$ and $x_2$ are points of nondifferentiability. We find that
\begin{align*}
x_1+\frac{1}{r}f_0'(x_1)&=-1-\frac{8}{r},&x_1+\frac{1}{r}f_1'(x_1)&=-1-\frac{4}{r},\\
x_2+\frac{1}{r}f_1'(x_2)&=1+\frac{8}{r},&x_2+\frac{1}{r}f_2'(x_2)&=1+\frac{15}{r}.\\
\end{align*}
Then according to Corollary \ref{pcubthm}, we have
$$P_rf(x)=\begin{cases}
p_1,&\mbox{if }x<-1-\frac{8}{r},\\
x_1,&\mbox{if }-1-\frac{8}{r}\leq x\leq-1-\frac{4}{r},\\
p_2,&\mbox{if }-1-\frac{4}{r}<x<1+\frac{8}{r},\\
x_2,&\mbox{if }1+\frac{8}{r}\leq x\leq1+\frac{15}{r},\\
p_3,&\mbox{if }1+\frac{15}{r}<x,
\end{cases}$$
where
\begin{align*}
p_1&=\frac{(4+r)-\sqrt{(4+r)^2+24(2-rx)}}{12},
p_2=\frac{-(6+r)+\sqrt{(6+r)^2+12(1+rx)}}{6},\\
p_3&=\frac{-(4+r)+\sqrt{(4+r)^2-36(2-rx)}}{18}.
\end{align*}
\begin{rem}\label{rem:choice}Note that in finding the proximal points of convex cubic functions, setting the derivative of the infimand of the Moreau envelope expression equal to zero and solving yields two points (positive and negative square root). However, the proximal mapping is strictly monotone and only one of the two points will be in the appropriate domain. The method for choosing the correct proximal point is laid out in Proposition \ref{prop:twopoints}.\end{rem}
\noindent As an illustration of Remark \ref{rem:choice}, consider our choice of $p_2$ above. The counterpart of $p_2$ has a negative square root, but $p_2$ is correct as given. It is easy to see that $p_2\in[x_1,x_2],$ by noting that $p_2(x)$ is an increasing function of $x$ and observing that $p_2(-1-4/r)=x_1$ and $p_2(1+8/r)=x_2.$ The proper choices of $p_1$ and $p_3$ are made in a similar manner. This method is presented in general form in Proposition \ref{prop:twopoints}. Then we have
$$e_rf(x)=\begin{cases}
-2p_1^3+2p_1^2+2p_1+3+\frac{r}{2}(p_1-x)^2,&\mbox{if }x<-1-\frac{8}{r},\\
5+\frac{r}{2}(-1-x)^2,&\mbox{if }-1-\frac{8}{r}\leq x\leq-1-\frac{4}{r},\\
p_2^3+3p_2^2-p_2+2+\frac{r}{2}(p_2-x)^2,&\mbox{if }-1-\frac{4}{r}<x<1+\frac{8}{r},\\
5+\frac{r}{2}(1-x)^2,&\mbox{if }1+\frac{8}{r}\leq x\leq1+\frac{15}{r},\\
3p_3^3+2p_3^2+2p_3-2+\frac{r}{2}(p_3-x)^2,&\mbox{if }1+\frac{15}{r}<x.
\end{cases}$$
\begin{figure}[H]\begin{center}
\includegraphics[scale=0.22]{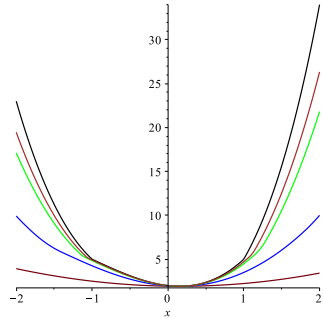}\end{center}
\caption{The functions $f(x)$ (black) and $e_rf(x)$ for $r=1,10,50,100.$}
\end{figure}
\begin{ex}\label{examplecubicquadratic}
Let $f:\R\rightarrow\R,$ $f(x)=\begin{cases}x^2,&\mbox{if }x\leq0,\\x^3,&\mbox{if }x>0.\end{cases}$\\
Then
$$P_rf(x)=\begin{cases}
\frac{rx}{r+2},&\mbox{if }x\leq0,\\\frac{-r+\sqrt{r^2+12rx}}{6},&\mbox{if }x>0,
\end{cases}$$ and
$$e_rf(x)=\begin{cases}
\frac{rx^2}{r+2},&\mbox{if }x\leq0,\\\left(\frac{-r+\sqrt{r^2+12rx}}{6}\right)^3+\frac{r}{2}\left(\frac{-r+\sqrt{r^2+12rx}}{6}-x\right)^2,&\mbox{if }x>0.
\end{cases}$$
\end{ex}
\begin{proof} The Moreau envelope is
\begin{align*}
e_rf(x)&=\inf\limits_{y\in\R}\left\{f(y)+\frac{r}{2}(y-x)^2\right\}\\
&=\min\left[\inf\limits_{y\leq0}\left\{y^2+\frac{r}{2}(y-x)^2\right\},\inf\limits_{y>0}\left\{y^3+\frac{r}{2}(y-x)^2\right\}\right].
\end{align*}
\noindent(i) Let $x\leq0.$ Then, with the restriction $y>0,$ $y^3+\frac{r}{2}(y-x)^2$ is minimized at $y=0,$ so that
$$\inf\limits_{y>0}\left\{y^3+\frac{r}{2}(y-x)^2\right\}=\frac{r}{2}x^2.$$ For the other infimum, setting the derivative of its argument equal to zero yields a minimizer of $y=\frac{rx}{r+2},$ so that
$$\inf\limits_{y\leq0}\left\{y^2+\frac{r}{2}(y-x)^2\right\}=\frac{rx^2}{r+2}=e_rf(x).$$
\noindent(ii) Let $x>0.$ Then, with the restriction $y\leq0,$ $y^2+\frac{r}{2}(y-x)^2$ is minimized at $y=0,$ so that
$$\inf\limits_{y\leq0}\left\{y^2+\frac{r}{2}(y-x)^2\right\}=\frac{r}{2}x^2.$$ For the other infimum, setting the derivative of its argument equal to zero yields a minimizer of $y=\frac{-r+\sqrt{r^2+12rx}}{6}$ (see Remark \ref{rem:choice}), so that
\begin{equation}\label{cubelessthan}
\inf\limits_{y>0}\left\{y^3+\frac{r}{2}(y-x)^2\right\}=\left(\frac{-r+\sqrt{r^2+12rx}}{6}\right)^3+\frac{r}{2}\left(\frac{-r+\sqrt{r^2+12rx}}{6}-x\right)^2,
\end{equation}
which is less than $\frac{r}{2}x^2$ for all $x>0.$ This can be seen by subtracting $\frac{r}{2}x^2$ from the right-hand side of \ref{cubelessthan}, and using calculus to show that the maximum of the resulting function is zero. The statement of the example follows.
\end{proof}
\noindent Figure \ref{cubesquare} illustrates the result for $r=1.$
\begin{figure}
\begin{center}\includegraphics[scale=0.22]{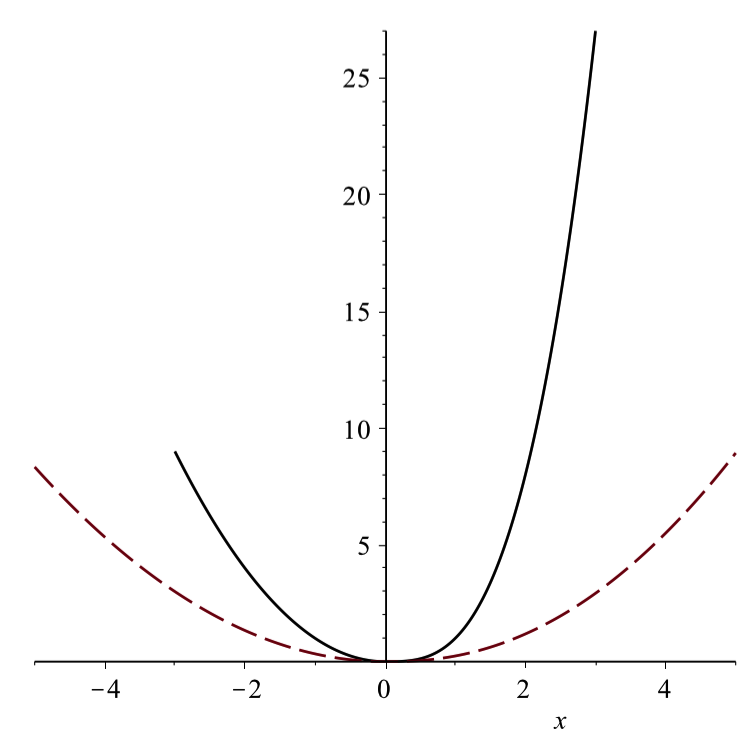}\end{center}
\caption{$f(x)$ (black), $e_1f(x)$ (red)}
\label{cubesquare}
\end{figure}
\noindent This result is perhaps surprising at first glance, since we know that as $r\nearrow\infty$ we must have $e_rf\nearrow f.$ This leads us to suspect that the Moreau envelope on the cubic portion of the function will be a cubic function, but the highest power of $x$ in the Moreau envelope is $2.$ The following proves that this envelope does indeed converge to $x^3.$ We have
\footnotesize\begin{align*}
&\lim\limits_{r\nearrow\infty}\left[\left(\frac{-r+\sqrt{r^2+12rx}}{6}\right)^3+\frac{1}{2}\left(\frac{-r+\sqrt{r^2+12rx}}{6}-x\right)^2\right]\\
=&\lim\limits_{r\nearrow\infty}\frac{-4r^3-36r^2x+6r^2+72rx+108x^2+(4r^2+12rx-6r-36x)\sqrt{r^2+12rx}}{216}\\
=&\lim\limits_{r\nearrow\infty}\frac{2x^3(4r^3-6r^2-27x)}{4r^3+36r^2x-6r^2-72rx-108x^2+(4r^2+12rx-6r-36x)\sqrt{r^2+12rx}}\\
=&\lim\limits_{r\nearrow\infty}\frac{2x^3\left(4-\frac{6}{r}-\frac{27x}{r^3}\right)}{4+\frac{36x}{r}-\frac{6}{r}-\frac{72x}{r^2}-\frac{108x^2}{r^3}+\left(4+\frac{12x}{r}-\frac{6}{r}-\frac{36x}{r^2}\sqrt{1+\frac{12x}{r}}\right)}\\
=&\frac{2x^3\cdot4}{4+4\sqrt{1}}=x^3.
\end{align*}\normalsize
\begin{ex}\label{cubeex}
Let $f:\R\rightarrow\R,$ $f(x)=|x|^3.$ Then
$$P_rf(x)=\begin{cases}
\frac{r-\sqrt{r^2-12rx}}{6},&\mbox{if }x<0,\\\frac{-r+\sqrt{r^2+12rx}}{6},&\mbox{if }x\geq0,
\end{cases}$$ and
$$e_rf(x)=\begin{cases}
\left(\frac{-r+\sqrt{r^2-12rx}}{6}\right)^3+\frac{r}{2}\left(\frac{r-\sqrt{r^2-12rx}}{6}-x\right)^2,&\mbox{if }x<0,\\
\left(\frac{-r+\sqrt{r^2+12rx}}{6}\right)^3+\frac{r}{2}\left(\frac{-r+\sqrt{r^2+12rx}}{6}-x\right)^2,&\mbox{if }x\geq0.
\end{cases}$$
\end{ex}
\begin{proof}The Moreau envelope is
\begin{align*}
e_rf(x)&=\inf\limits_{y\in\R}\left\{f(y)+\frac{r}{2}(y-x)^2\right\}\\
&=\min\left[\inf\limits_{y\leq0}\left\{-y^3+\frac{r}{2}(y-x)^2\right\},\inf\limits_{y>0}\left\{y^3+\frac{r}{2}(y-x)^2\right\}\right].
\end{align*}
By an argument identical to that of the previous example, we find that for $x\geq0,$
$$e_rf(x)=\left(\frac{-r+\sqrt{r^2+12rx}}{6}\right)^3+\frac{r}{2}\left(\frac{-r+\sqrt{r^2+12rx}}{6}-x\right)^2.$$
Then by Lemma \ref{evenlem}, we conclude the statement of the example.
\end{proof}
\begin{ex}
Let $f:\R\rightarrow\R,$ $f(x)=|x|^3+ax.$ Then
$$P_rf(x)=\begin{cases}
\frac{r-\sqrt{r^2-12(rx-a)}}{6},&\mbox{if }x<\frac{a}{r},\\\frac{-r+\sqrt{r^2+12(rx-a)}}{6},&\mbox{if }x\geq\frac{a}{r},
\end{cases}$$ and \footnotesize
$$e_rf(x)=\begin{cases}\left(\frac{-r+\sqrt{r^2-12(rx-a)}}{6}\right)^3+\frac{r}{2}\left(\frac{-r+\sqrt{r^2-12(rx-a)}}{6}+x-\frac{a}{r}\right)^2+ax-\frac{a^2}{2r},&\mbox{if }x<\frac{a}{r}\\
\left(\frac{-r+\sqrt{r^2+12(xr-a)}}{6}\right)^3+\frac{r}{2}\left(\frac{-r+\sqrt{r^2+12(xr-a)}}{6}-x+\frac{a}{r}\right)^2+ax-\frac{a^2}{2r},&\mbox{if }x\geq\frac{a}{r}.
\end{cases}$$\normalsize
\end{ex}
\begin{proof}
The proof is found by applying Fact \ref{affineshift} to Example \ref{cubeex}.\end{proof}

\subsection{Main result}

The examples of the previous section suggest a theorem for the case of a general convex cubic function on $\R$. The theorem is the following.
\begin{thm}\label{thm:pcf} Let $f:\R\rightarrow\R,$ $f(x)=a|x|^3+bx^2+cx+d,$ with $a,b\geq0.$ Define
\begin{align*}
p_1&=\frac{r+2b-\sqrt{(r+2b)^2-12a(rx-c)}}{6a},\\
p_2&=\frac{-r-2b+\sqrt{(r+2b)^2+12a(rx-c)}}{6a}.\end{align*}
Then the proximal mapping and Moreau envelope of $f$ are
\begin{align*}
P_rf(x)&=\begin{cases}
p_1,&\mbox{\emph{if} }x<\frac{c}{r},\\p_2,&\mbox{\emph{if} }x\geq\frac{c}{r},
\end{cases}\\
e_rf(x)&=\begin{cases}
-ap_1^3+bp_1^2+d-p_1(rx-c)+\frac{r}{2}(p_1^2+x^2),&\mbox{\emph{if} }x<\frac{c}{r},\\
ap_2^3+bp_2^2+d-p_2(rx-c)+\frac{r}{2}(p_2^2+x^2),&\mbox{\emph{if} }x\geq\frac{c}{r}.
\end{cases}\end{align*}
\end{thm}
\begin{proof}
We first consider $g(x)=a|x|^3+bx^2+d,$ and we use Lemma \ref{affineshift} to account for the $cx$ term later. By the same method as in Example \ref{examplecubicquadratic}, for $x<0$ we find that $$q_1=P_rg(x)=\frac{r+2b-\sqrt{(r+2b)^2-12arx}}{6a}$$ and
$$e_rg(x)=-aq_1^3+bq_1^2+d+\frac{r}{2}(q_1-x)^2.$$ Then by Lemma \ref{evenlem}, for $x\geq0$ we have that $$q_2=P_rg(x)=\frac{-r-2b+\sqrt{(r+2b)^2+12arx}}{6a}$$ and
$$e_rg(x)=aq_2^3+bq_2^2+d+\frac{r}{2}(q_2-x)^2.$$
Finally, Lemma \ref{affineshift} gives us that $e_rf(x)=e_rg\left(x-\frac{c}{r}\right)+cx-\frac{c^2}{2r},$ which yields the proximal mapping and Moreau envelope that we seek.
\end{proof}
\noindent Now we present the application of Corollary \ref{pcubthm} to convex piecewise cubic functions. First, we deal with the issue mentioned in Remark \ref{rem:choice}: making the proper choice of proximal point for a cubic piece.
\begin{prop}\label{prop:twopoints}
Let $f:\R\to\R$ be a convex piecewise cubic function (see Lemma \ref{lem:piecewisecubic}), with each piece $f_i$ defined by
$$f_i(x)=a_ix^3+b_ix^2+c_ix+d_i,\forall x\in\R.$$ Then on each subdomain $S_i=\left[x_i+\frac{1}{r}f_i'(x_i),x_{i+1}+\frac{1}{r}f_i'(x_{i+1})\right]$ (and setting $x_0=-\infty$ and $x_{m+1}=\infty$), the proximal point of $f_i$ is
$$ p_i=\frac{-(2b_i+r)+\sqrt{(2b_i+r)^2-12a_i(c_i-rx)}}{6a_i}.$$
\end{prop}
\begin{proof}
Recall from Lemma \ref{lem:piecewisecubic} that $\dom f_i=\R$ for each $i$. For $f_i,$ the proximal mapping is
$$P_rf_i(x)=\argmin\limits_{y\in\R}\left\{a_iy^3+b_iy^2+c_iy+d_i+\frac{r}{2}(y-x)^2\right\}.$$
Setting the derivative of the infimand equal to zero yields the potential proximal points:
$$3a_iy^2+2b_iy+c_i+ry-rx=0=3a_iy^2+(2b_i+r)y+(c_i-rx),$$
\begin{equation}\label{eq1}y=\frac{-(2b_i+r)\pm\sqrt{(2b_i+r)^2-12a_i(c_i-rx)}}{6a_i}.\end{equation} Notice that any $x\in S_i$ can be written as
$$x=\tilde{x}+\frac{1}{r}(3a_i\tilde{x}^2+2b_i\tilde{x}+c_i)\mbox{ for some }\tilde{x}\in[x_i,x_{i+1}].$$Substituting into \eqref{eq1} yields
\begin{equation}\label{plusminus}y=\frac{-(2b_i+r)\pm|2b_i+r+6a_i\tilde{x}|}{6a_i}\end{equation}
Since $f_i$ is convex on $[x_i,x_{i+1}],$ the second derivative is nonnegative: $6a_i\tilde{x}+2b_i\geq0$ for all $\tilde{x}\in[x_i,x_{i+1}].$ Thus, $|2b_i+r+6a_i\tilde{x}|=2b_i+r+6a_i\tilde{x}$ and the two points of \eqref{plusminus} are
\begin{align*}
p_i&=\frac{-(2b_i+r)+2b_i+r+6a_i\tilde{x}}{6a_i}=\tilde{x},\\
p_j&=\frac{-(2b_i+r)-2b_i-r-6a_i\tilde{x}}{6a_i}=-\tilde{x}-\frac{2b_i+r}{3a_i}.\end{align*}
Therefore, $p_i$ is the proximal point, since it lies in $[x_i,x_{i+1}].$ This corresponds to the positive square root of \eqref{eq1}, which gives us the statement of the proposition.
\end{proof}
\begin{cor}\label{cor:pc2}
Let $f:\R\to\R$ be a convex piecewise cubic function:
$$f(x)=\begin{cases}
f_0(x),&\mbox{\emph{if} }x\leq x_1,\\
f_1(x),&\mbox{\emph{if} }x_1\leq x\leq x_2,\\
&\vdots\\
f_m(x),&\mbox{\emph{if} }x_m\leq x,
\end{cases}$$
where
$$f_i(x)=a_ix^3+b_ix^2+c_ix+d_i,\qquad a_i,b_i,c_i,d_i\in\R.$$For each $i\in\{0,1,\ldots,m\},$ define
$$p_i=\frac{-(2b_i+r)+\sqrt{(2b_i+r)^2-12a_i(c_i-rx)}}{6a_i}.$$ Partition $\dom f$ as follows:
\begin{align*}
S_0&=\left(-\infty,x_1+\frac{1}{r}(3a_0x_1^2+2b_0x_1+c_0)\right),\\
S_1&=\left[x_1+\frac{1}{r}(3a_0x_1^2+2b_0x_1+c_0),x_1+\frac{1}{r}(3a_1x_1^2+2b_1x_1+c_1)\right],\\
S_2&=\left(x_1+\frac{1}{r}(3a_1x_1^2+2b_1x_1+c_1),x_2+\frac{1}{r}(3a_1x_2^2+2b_1x_2+c_1)\right),\\
&\vdots\\
S_{2m}&=\left(x_m+\frac{1}{r}(3a_mx_m^2+2b_mx_m+c_m),\infty\right).
\end{align*}
Then the proximal mapping and Moreau envelope of $f$ are
$$P_rf(x)=\begin{cases}
p_0,&\mbox{\emph{if} }x\in S_0,\\
x_1,&\mbox{\emph{if} }x\in S_1,\\
p_1,&\mbox{\emph{if} }x\in S_2,\\
&\vdots\\
p_m,&\mbox{\emph{if} }x\in S_{2m},
\end{cases} e_rf(x)=\begin{cases}
f_0(p_0)+\frac{r}{2}(p_0-x)^2,&\mbox{\emph{if} }x\in S_0,\\
f_1(x_1)+\frac{r}{2}(x_1-x)^2,&\mbox{\emph{if} }x\in S_1,\\
f_1(p_1)+\frac{r}{2}(p_1-x)^2,&\mbox{\emph{if} }x\in S_2,\\
&\vdots\\
f_m(p_m)+\frac{r}{2}(p_m-x)^2,&\mbox{\emph{if} }x\in S_{2m}.
\end{cases}$$
\end{cor}
\noindent Algorithm \ref{alg:piecewisecubic} below is a block of pseudocode that accepts as input a set of $m$ cubic functions $\{f_1,\ldots,f_m\}$ and $m-1$ intersection points $\{x_1,\ldots,x_{m-1}\}$ that form the convex piecewise cubic function $f,$ calculates $e_rf$ and plots $f$ and $e_rf$ together.
\begin{algorithm}[H]
\caption{: A routine for graphing the Moreau envelope of a convex piecewise cubic function.}
\label{alg:piecewisecubic}
\begin{algorithmic}
\STATE \textbf{Step 0.} Input coefficients of $f_i,$ intersection points, prox-parameter $r,$ lower and upper bounds for the graph.
\STATE \textbf{Step 1.} Find $f_i'$ for each $i.$
\STATE \textbf{Step 2.} Use $f_i'$ and $x_i$ to define the subdomains $S_i$ of $e_rf$ as found in Corollary \ref{cor:pc2}.
\STATE \textbf{Step 3.} On each $S_i,$ use Proposition \ref{prop:twopoints} to find the proximal point $p_i.$
\STATE \textbf{Step 4.} Find $e_rf(x)=f_i(p)+\frac{r}{2}(p-x)^2$ for $x\in S_i,$ where $p$ is $p_i$ for $i$ even and $x_i$ for $i$ odd.
\STATE \textbf{Step 5.} Plot $f$ and $e_rf$ on the same axes.
\end{algorithmic}
\end{algorithm}

\section{Smoothing a gauge function via the Moreau envelope}\label{sec:anf}

In this section, we focus on the idea of smoothing a gauge function. Gauge functions are proper, lsc and convex, but many gauge functions have ridges of nondifferentiability that can be regularized by way of the Moreau envelope. The main result of this section is a method of smoothing a gauge function that yields another gauge function that is differentiable everywhere except on the kernel, as we shall see in Theorem \ref{normthm}. A special case of a gauge function is a norm function; Corollary \ref{normcor} applies Theorem \ref{normthm} to an arbitrary norm function, resulting in another norm function that is smooth everywhere except at the origin.\par To our knowledge, this smoothing of gauge functions and norm functions is a new development in Convex Optimization. It is our hope that this new theory will be of interest and of some practical use to the readers of this paper.

\subsection{Definitions}

We begin with some definitions that are used only in this section.
\begin{defn}
For $x\in\R,$ the \emph{sign function} $\sgn(x)$ is defined
$$\sgn(x)=\begin{cases}
1,&\mbox{if }x>0,\\0,&\mbox{if }x=0,\\-1,&\mbox{if }x<0.
\end{cases}$$
\end{defn}
\begin{defn}
A function $k$ on $\R^n$ is a \emph{gauge} if $k$ is a nonnegative, positively homogeneous, convex function such that $k(0)=0.$
\end{defn}
\begin{defn}
A function $f$ on $\R^n$ is \emph{gauge-like} if $f(0)=\inf f$ and the lower level sets
$$\{x:f(x)\leq\alpha\},~f(0)<\alpha<\infty$$
are all proportional, i.e. they can all be expressed as positive scalar multiples of a single set.
\end{defn}
\noindent Note that any norm function is a closed gauge. Theorem \ref{thm:gauge} below gives us a way to construct gauge-like functions that are not necessarily gauges.
\begin{thm}\emph{\cite[Theorem 15.3]{convanalrock}}\label{thm:gauge}
A function $f$ is a gauge-like closed proper convex function if and only if it can be expressed in the form$$f(x)=g(k(x)),$$where $k$ is a closed gauge and $g$ is a nonconstant nondecreasing lsc convex function on $[0,\infty]$ such that $g(y)$ is finite for some $y>0$ and $g(\infty)=\infty.$ If $f$ is gauge-like, then $f^*$ is gauge-like as well.
\end{thm}
\begin{ex}
Let $k:\R\to\R,$ $k(x)=|x|$ and $g:[0,\infty]\to\R,$ $g(x)=x+1.$ Then by Theorem \ref{thm:gauge}, we have that$$f(x)=g(k(x))=|x|+1$$is gauge-like, and so is $$f^*(y)=\begin{cases}-1,&\mbox{if }-1\leq y\leq1,\\\infty,&\mbox{otherwise.}\end{cases}$$
\end{ex}

\subsection{Main result and illustrations}

The following theorem and corollary are the main results of this section. Then some typical norm functions on $\R^2$ are showcased: the $\infty$-norm and the $\ell^1$-norm. Finally, by way of counterexample we demonstrate that the Moreau envelope is ideal for the smoothing effect of Theorem \ref{normthm} and other regularizations may not be; the Pasch-Hausdorff envelope is shown not to have the desired effect.
\begin{thm}\label{normthm}
Let $f:\R^n\rightarrow\overline{\R}$ be a gauge function. Define $g_r(x)=[e_r(f^2)](x)$ and $h_r=\sqrt{g_r}.$ Then $h_r$ is a gauge function, differentiable except on $\{x: f(x)=0\},$ and $\lim\limits_{r\nearrow\infty}h_r=f.$
\end{thm}
\begin{proof}
By \cite[Theorem 1.25]{rockwets}, $\lim_{r\nearrow\infty}g_r=f^2.$ So we have $\lim_{r\nearrow\infty}h_r=|f|,$ which is simply $f$ since $f(x)\geq0$ for all $x.$
Since $f$ is nonnegative and convex, $f^2$ is proper, lsc and convex. By \cite[Theorem 2.26]{rockwets} we have that $g_r$ is convex and continuously differentiable everywhere, the gradient being
$$\nabla g_r(x)=r[x-P_rf^2(x)].$$
Then by the chain rule, we have that
$$\nabla h_r(x)=\frac{1}{2}[g_r(x)]^{-\frac{1}{2}}r[x-P_rf^2(x)], \text{ provided that } g_r(x)\neq0.$$
Since $\inf g_{r}=\inf f^2 =0$ and $\argmin g_{r}=\argmin f^2$, we have $g_{r}(x)=0$ if and only if $f^2(x)=0$, i.e., $f(x)=0$.
To see that $h_r$ is a gauge function, we have
\begin{align*}
g_r(\alpha x)&=[e_r(f^2)](\alpha x)=\inf\limits_{y\in\R^n}\left\{f^2(y)+\frac{r}{2}\|y-\alpha x\|^2\right\}\\
&=\inf\limits_{y\in\R^n}\left\{f^2(y)+\frac{\alpha^2r}{2}\left\|\frac{y}{\alpha}-x\right\|^2\right\}\\
&=\alpha^2\inf\limits_{\frac{y}{\alpha}\in\R^n}\left\{f^2\left(\frac{y}{\alpha}\right)+\frac{r}{2}\left\|\frac{y}{\alpha}-x\right\|^2\right\}\\
&=\alpha^2\inf\limits_{\tilde{y}\in\R^n}\left\{f^2(\tilde{y})+\frac{r}{2}\|\tilde{y}-x\|^2\right\}=\alpha^2g_r(x).
\end{align*}
Thus, $g_r$ is positively homogeneous of degree two. Hence, by \cite[Corollary 15.3.1]{convanalrock}, there exists a closed gauge function $k$ such that $g_r(x)=\frac{1}{2}k^2(x).$ Then
$$h_r(x)=\sqrt{g_r(x)}=\sqrt{\frac{1}{2}k^2(x)}=\frac{1}{\sqrt{2}}k(x)$$
and we have that $h_r$ is a gauge function.
\end{proof}
\begin{cor}\label{normcor}
Let $f:\R^n\rightarrow\overline{\R},$ $f(x)=\|x\|_*$ be an arbitrary norm function. Define $g_r(x)=[e_r(f^2)](x)$ and $h_r=\sqrt{g_r}.$ Then $h_r$ is a norm, $h_r$ is differentiable everywhere except at the origin, and $\lim\limits_{r\nearrow\infty}h_r=f.$
\end{cor}
\begin{proof}
By Theorem \ref{normthm}, we have that $\lim_{r\nearrow\infty}h_r=f,$ $h_r$ is differentiable everywhere except at the origin, $h_r$ is nonnegative and positively homogeneous. To see that $h_r$ is a norm, it remains to show that
\begin{itemize}
\item[(i)] $h_r(x)=0\Rightarrow x=0$ and
\item[(ii)] $h_r(x+y)\leq h_r(x)+h_r(y)$ for all $x,y\in\R^n.$
\end{itemize}
\noindent (i) Suppose that $h_r(x)=0.$ Then
\begin{align*}
\sqrt{[e_r(f^2)](x)}&=0,\\
\inf\limits_{y\in\R^n}\left\{\|y\|_*^2+\frac{r}{2}\|y-x\|^2\right\}&=0,\\
\|\tilde{y}\|^2_*+\|\tilde{y}-x\|^2&=0\mbox{ for some }\tilde{y},\mbox{ (since }\|\cdot\|_*^2\mbox{ is strongly convex)}\\
\|\tilde{y}\|_*&=-\|\tilde{y}-x\|\Rightarrow\tilde{y}=0\Rightarrow x=0.
\end{align*}
\noindent(ii) We have that $h_r$ is convex, since it is a gauge function. Therefore, by \cite[Theorem 4.7]{convanalrock}, the triangle inequality holds.
\end{proof}
\noindent Now we present some examples on $\R^2,$ to illustrate the method of Theorem \ref{normthm}.
\begin{ex}\label{ex:maxnorm}
Let $f:\R^2\rightarrow\R,$ $f(x,y)=\max(|x|,|y|).$ Define $g_r(x,y)=[e_r(f^2)](x,y)$ and $h_r(x,y)=\sqrt{g_r(x,y)}.$ Then, with $\R^2$ partitioned as
\begin{align*}
R_1^r&=\left\{(x,y):-\frac{r}{r+2}x\leq y\leq\frac{r}{r+2}x\right\}\cup\left\{(x,y):\frac{r}{r+2}x\leq y\leq-\frac{r}{r+2}x\right\},\\
R_2^r&=\left\{(x,y):-\frac{r}{r+2}y\leq x\leq\frac{r}{r+2}y\right\}\cup\left\{(x,y):\frac{r}{r+2}y\leq x\leq-\frac{r}{r+2}y\right\},\\
R_3^r&=\left\{(x,y):\frac{r}{r+2}y\leq x\leq\frac{r+2}{r}y\right\}\cup\left\{(x,y):\frac{r+2}{r}y\leq x\leq\frac{r}{r+2}y\right\},\\
R_4^r&=\left\{(x,y):-\frac{r}{r+2}y\leq x\leq-\frac{r+2}{r}y\right\}\cup\left\{(x,y):-\frac{r+2}{r}y\leq x\leq-\frac{r}{r+2}y\right\},
\end{align*}
we have
$$P_rh_r(x,y)=\begin{cases}
\left(\frac{rx}{r+2},y\right),&\mbox{if }(x,y)\in R_1^r,\\
\left(x,\frac{ry}{r+2}\right),&\mbox{if }(x,y)\in R_2^r,\\
\left(\frac{r(x+y)}{2(r+1)},\frac{r(x+y)}{2(r+1)}\right),&\mbox{if }(x,y)\in R_3^r,\\
\left(\frac{r(x-y)}{2(r+1)},\frac{-r(x-y)}{2(r+1)}\right),&\mbox{if }(x,y)\in R_4^r,
\end{cases}$$
$$h_r(x,y)=\begin{cases}
\sqrt{\frac{r}{r+2}}|x|,&\mbox{if }(x,y)\in R_1^r,\\
\sqrt{\frac{r}{r+2}}|y|,&\mbox{if }(x,y)\in R_2^r,\\
\sqrt{\frac{r^2(x-y)^2+2r(x^2+y^2)}{4(r+1)}},&\mbox{if }(x,y)\in R_3^r,\\
\sqrt{\frac{r^2(x+y)^2+2r(x^2+y^2)}{4(r+1)}},&\mbox{if }(x,y)\in R_4^r,
\end{cases}$$
and $\lim\limits_{r\nearrow\infty}h_r=f,$ $\lim\limits_{r\searrow0}h_r=0.$
\end{ex}
\begin{proof}
Figure \ref{fig:xy} shows the partitioning of $\R^2$ in the case $r=1;$ for other values of $r$ the partition is of similar form.
\begin{figure}[H]
\begin{center}\includegraphics[scale=0.21]{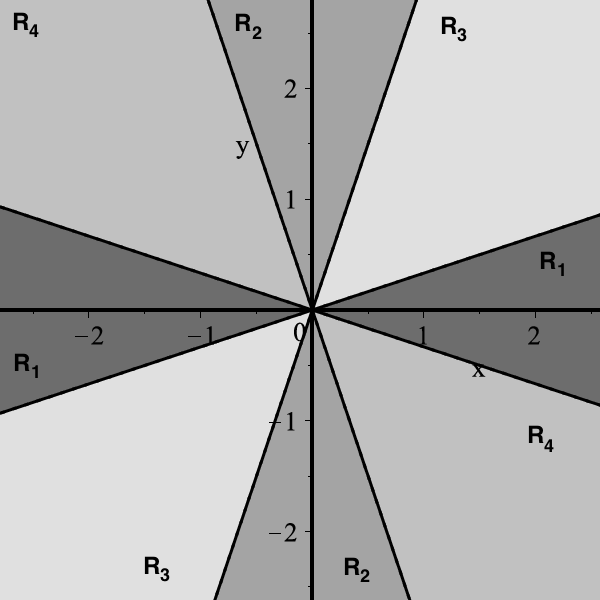}\end{center}
\caption{The four regions of the piecewise function $h_1(x,y).$}
\label{fig:xy}
\end{figure}
\noindent We have\small
\begin{align}
&[e_r(f^2)](\bar{x},\bar{y})\nonumber\\
=&\inf\limits_{(x,y)\in\R^2}\left\{[\max\{|x|,|y|\}]^2+\frac{r}{2}\left[(x-\bar{x})^2+(y-\bar{y})^2\right]\right\}\nonumber\\
=&\min\left[\inf\limits_{|x|\geq|y|}\left\{x^2+\frac{r}{2}\left[(x-\bar{x})^2+(y-\bar{y})^2\right]\right\},\inf\limits_{|x|<|y|}\left\{y^2+\frac{r}{2}\left[(x-\bar{x})^2+(y-\bar{y})^2\right]\right\}\right].\label{biggerone}
\end{align}\normalsize
We denote the first infimum of \eqref{biggerone} by $I_x,$ and the second one by $I_y.$ We need to split the restriction on $I_x$ into three pieces: the differentiable portion $|x|>|y|,$ the $x=y$ portion, and the $x=-y$ portion. We denote these three infima as $I_{|x|>|y|},$ $I_{x=y},$ and $I_{x=-y}.$ Similarly, we split $I_y$ into $I_{|x|<|y|},$ $I_{x=y},$ and $I_{x=-y}.$ Considering $I_{|x|>|y|},$ we set the gradient of its argument equal to zero and find a proximal point of $(x,y)=\left(\frac{r\bar{x}}{r+2},\bar{y}\right),$ which yields an infimum of
$$I_{|x|>|y|}=\frac{r\bar{x}^2}{r+2}.$$
This is the result for $|x|>|y|,$ or in other words for $\left|\frac{r\bar{x}}{r+2}\right|>|\bar{y}|$ (region $R_1^r$). In a moment we compare this result to $I_{x=y}$ and $I_{x=-y};$ $I_x$ is the minimum of the three. By a symmetric process, considering $I_{|x|<|y|}$ we find a proximal point of $\left(\bar{x},\frac{r\bar{y}}{r+2}\right).$ This gives
$$I_{|x|<|y|}=\frac{r\bar{y}^2}{r+2}$$
for $|x|<|y|,$ or in other words for $|\bar{x}|<\left|\frac{r\bar{y}}{r+2}\right|$ (region $R_2^r$). It is clear that if $|x|>|y|,$ then we have $I_{|x|>|y|}<I_{|x|<|y|},$ and if $|x|<|y|,$ then $I_{|x|<|y|}<I_{|x|>|y|}.$ Hence, $g_r(\bar{x},\bar{y})$ will be $\min(I_{|x|>|y|},I_{x=y},I_{x=-y})$ on $R_1^r$ and $\min(I_{|x|<|y|},I_{x=y},I_{x=-y})$ on $R_2^r.$ Now we consider $I_{x=y}.$ In this case, the infimum reduces to a one-dimensional problem, $$\inf\limits_{x\in\R}\left\{x^2+\frac{r}{2}\left[(x-\bar{x})^2+(x-\bar{y})^2\right]\right\},$$ whose solution is
$$I_{x=y}=\frac{r^2(\bar{x}-\bar{y})^2+2r(\bar{x}^2+\bar{y}^2)}{4(r+1)},\mbox{ with proximal point }\left(\frac{r(\bar{x}+\bar{y})}{2(r+1)},\frac{r(\bar{x}+\bar{y})}{2(r+1)}\right).$$
Similarly, we find that
$$I_{x=-y}=\frac{r^2(\bar{x}+\bar{y})^2+2r(\bar{x}^2+\bar{y}^2)}{4(r+1)},\mbox{ with proximal point }\left(\frac{r(\bar{x}-\bar{y})}{2(r+1)},\frac{-r(\bar{x}-\bar{y})}{2(r+1)}\right).$$
Now let us compare $I_{|x|>|y|}$ to $I_{x=y}.$ We show that on $R_1^r$ the latter is always greater than or equal to the former, by assuming so and arriving at a tautology:
\begin{align*}
\frac{r^2(\bar{x}-\bar{y})^2+2r(\bar{x}^2+\bar{y}^2)}{4(r+1)}&\geq\frac{r\bar{x}^2}{r+2}\\
(r+2)[r^2(\bar{x}^2-2\bar{x}\bar{y}+\bar{y}^2)+2r\bar{x}^2+2r\bar{y}^2]&\geq4r(r+1)\bar{x}^2\\
r[(r+2)^2\bar{x}^2-2r(r+2)\bar{x}\bar{y}+(r+2)^2\bar{y}^2]&\geq r(4r+4)\bar{x}^2\\
(r^2+4r+4-4r-4)\bar{x}^2-2r(r+2)\bar{x}\bar{y}+(r+2)^2\bar{y}^2&\geq0\\
r^2\bar{x}^2-2r(r+2)\bar{x}\bar{y}+(r+2)^2\bar{y}^2&\geq0\\
[r\bar{x}-(r+2)\bar{y}]^2&\geq0.
\end{align*}
Thus, $I_{x=y}\geq I_{|x|>|y|}$ on $R_1^r.$ By identical arguments, one can show that $I_{x=-y}\geq I_{|x|>|y|}$ on $R_1^r,$ and that $I_{x=y}\geq I_{|x|<|y|}$ and $I_{x=-y}\geq I_{|x|<|y|}$ on $R_2^r.$ Therefore, we have $g_r(x,y)=I_{|x|>|y|}$ on $R_1^r$ and $g_r(x,y)=I_{|x|<|y|}$ on $R_2^r.$ On $R_3^r$ and $R_4^r,$ the Moreau envelope is $\min(I_{x=y},I_{x=-y}),$ since $I_{|x|>|y|}$ and $I_{|x|<|y|}$ are not valid outside of $\left|\frac{rx}{r+2}\right|>|y|$ and $|x|<\left|\frac{ry}{r+2}\right|,$ respectively. Notice that comparing $I_{x=y}$ with $I_{x=-y}$ is equivalent to comparing $(\bar{x}-\bar{y})^2$ with $(\bar{x}+\bar{y})^2,$ which reduces to comparing $-\bar{x}\bar{y}$ with $\bar{x}\bar{y}.$ It is clear that $-\bar{x}\bar{y}<\bar{x}\bar{y}$ if and only if $\sgn(\bar{x})=\sgn(\bar{y})=\pm1.$ Thus, $g_r(x,y)=I_{x=y}$ on the region outside of $R_1^r\cup R_2^r$ where $x,y>0$ and where $x,y<0,$ which is $R_3^r.$ Similarly, $g_r(x,y)=I_{x=-y}$ outside of $R_1^r\cup R_2^r$ where $x>0,y<0$ and where $x<0,y>0,$ which is $R_4^r.$ Therefore, the proximal mapping of $g_r$ is
$$P_rg_r(x,y)=\begin{cases}
\left(\frac{rx}{r+2},y\right),&\mbox{if }(x,y)\in R_1^r,\\
\left(x,\frac{ry}{r+2}\right),&\mbox{if }(x,y)\in R_2^r,\\
\left(\frac{r(x+y)}{2(r+1)},\frac{r(x+y)}{2(r+1)}\right),&\mbox{if }(x,y)\in R_3^r,\\
\left(\frac{r(x-y)}{2(r+1)},\frac{-r(x-y)}{2(r+1)}\right),&\mbox{if }(x,y)\in R_4^r.
\end{cases}$$
Applying to \eqref{biggerone}, we find that
$$g_r(x,y)=\begin{cases}
\frac{rx^2}{r+2},&\mbox{if }(x,y)\in R_1^r,\\
\frac{ry^2}{r+2},&\mbox{if }(x,y)\in R_2^r,\\
\frac{r^2(x-y)^2+2r(x^2+y^2)}{4(r+1)},&\mbox{if }(x,y)\in R_3^r,\\
\frac{r^2(x+y)^2+2r(x^2+y^2)}{4(r+1)},&\mbox{if }(x,y)\in R_4^r.
\end{cases}$$
Finally, $h_r=\sqrt{g_r}$ has the same proximal mapping as $g_r,$ so
$h_r(x,y)$ is as stated in the example.
\begin{figure}
\begin{center}\includegraphics[scale=0.35]{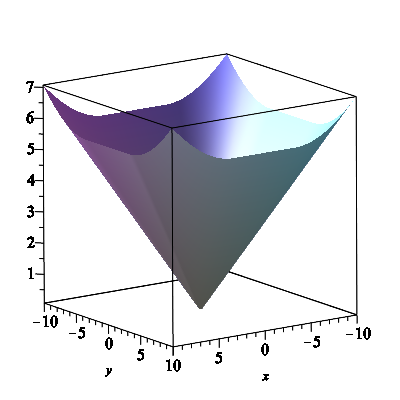}\end{center}
\caption{The function $h_1(x,y).$}
\label{fig:hxy}
\end{figure}
\noindent Now let us take a look at what happens to $h_r$ when $r\nearrow\infty.$ By Theorem \ref{normthm}, we expect to recover $f.$ Taking the limit of $R_1^r,$ we have
\begin{align*}
\lim\limits_{r\nearrow\infty}{R_1^r}&=\lim\limits_{r\nearrow\infty}\left[\left\{(x,y):\frac{-r}{r+2}x\leq y\leq\frac{rx}{r+2}\right\}\cup\left\{(x,y):\frac{rx}{r+2}\leq y\leq\frac{-rx}{r+2}\right\}\right],\\
&=\{(x,y):-x\leq y\leq x\}\cup\{(x,y):x\leq y\leq-x\}\\
&=\{(x,y):|x|\geq|y|\}.
\end{align*}
Similarly, we find that
\begin{align*}
\lim\limits_{r\nearrow\infty}{R_2^r}&=\{(x,y):|x|\leq|y|\},\\
\lim\limits_{r\nearrow\infty}{R_3^r}&=\{(x,y):x=y\},\\
\lim\limits_{r\nearrow\infty}{R_4^r}&=\{(x,y):x=-y\}.\end{align*}
Since $R_3^r$ and $R_4^r$ are now contained in $R_1^r,$ we need consider the limit of $h_r$ over $R_1^r$ and $R_2^r$ only. Therefore,
\begin{align*}
\lim\limits_{r\nearrow\infty}h_r(x,y)&=\begin{cases}
\lim\limits_{r\nearrow\infty}\sqrt{\frac{r}{r+2}}|x|,&\mbox{if }(x,y)\in R_1^r,\\
\lim\limits_{r\nearrow\infty}\sqrt{\frac{r}{r+2}}|y|,&\mbox{if }(x,y)\in R_2^r,
\end{cases}\\
&=\begin{cases}
|x|,&\mbox{if }|x|\geq|y|,\\
|y|,&\mbox{if }|x|\leq|y|,
\end{cases}\\
&=\max\{|x|,|y|\}=f(x,y).
\end{align*}
If, on the other hand, we take the limit as $r$ goes down to zero, then it is $R_3^r$ and $R_4^r$ that become all of $\R^2,$ with
\begin{align*}
\lim\limits_{r\searrow0}{R_3^r}&=\{(x,y):x,y\geq0\}\cup\{(x,y):x,y\leq0\},\\
\lim\limits_{r\searrow0}{R_4^r}&=\{(x,y):x\geq0,y\leq0\}\cup\{(x,y):x\leq0,y\geq0\}.
\end{align*}
Then $R_1^r$ and $R_2^r$ are contained in $R_3^r,$ and the limit of $h_r$ is
\begin{align*}
\lim\limits_{r\searrow0}h_r(x,y)&=\begin{cases}
\lim\limits_{r\searrow0}\sqrt{\frac{r^2(x-y)^2+2r(x^2+y^2)}{4(r+1)}},&\mbox{if }(x,y)\in R_3^r,\\
\lim\limits_{r\searrow0}\sqrt{\frac{r^2(x+y)^2+2r(x^2+y^2)}{4(r+1)}},&\mbox{if }(x,y)\in R_4^r
\end{cases}\\
&=\begin{cases}
\sqrt{\frac{0}{4}},&\mbox{if }x,y\geq0\mbox{ \textbf{or} }x,y\leq0,\\
\sqrt{\frac{0}{4}},&\mbox{if }x\geq0,y\leq0\mbox{ \textbf{or} }x\leq0,y\geq0
\end{cases}\\
&=0.\qedhere
\end{align*}
\end{proof}
\noindent Figure \ref{fig:h_max} shows the graphs of $h_r$ for several values of $r,$ and demonstrates the effect of $r\searrow0$ and $r\nearrow\infty.$
\begin{figure}[ht]
\begin{center}\includegraphics[scale=0.35]{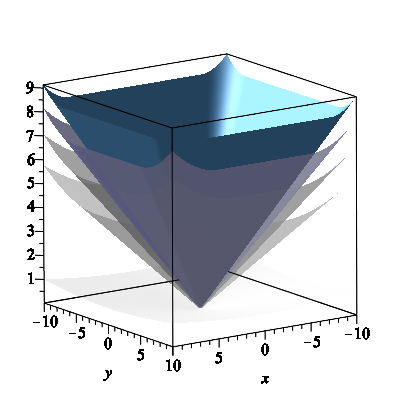}\end{center}
\caption{The function $h_r(x,y)$ from $r=0.01$ (grey) to $r=5$ (blue).}
\label{fig:h_max}
\end{figure}
\begin{ex}
Let $f:\R^2\rightarrow\R,$ $f(x,y)=|x|+|y|.$ Define $g_r(x,y)=[e_r(f^2)](x,y)$ and $h_r=\sqrt{g_r}.$ Then, with $\R^2$ partitioned as
\begin{align*}
R_1^r=&\left\{(x,y):\frac{2}{r+2}y\leq x\leq\frac{r+2}{2}y\right\}\cup\left\{(x,y):\frac{r+2}{2}y\leq x\leq\frac{2}{r+2}y\right\},\\
R_2^r=&\left\{(x,y):-\frac{2}{r+2}y\leq x\leq-\frac{r+2}{2}y\right\}\cup\left\{(x,y):-\frac{r+2}{2}y\leq x\leq-\frac{2}{r+2}\right\},\\
R_3^r=&\left\{(x,y):\frac{r+2}{2}y\leq x,-\frac{r+2}{2}y\leq x\right\}\cup\left\{(x,y):\frac{r+2}{2}y\geq x,-\frac{r+2}{2}y\geq x\right\}\\
R_4^r=&\left\{(x,y):\frac{r+2}{2}x\geq y,-\frac{r+2}{2}x\geq y\right\}\cup\left\{(x,y):\frac{r+2}{2}x\leq y,-\frac{r+2}{2}x\leq y\right\},
\end{align*}
we have
$$P_rh_r(x,y)=\begin{cases}
\left(\frac{(r+2)x-2y}{r+4},\frac{-2x+(r+2)y}{r+4}\right),&\mbox{if }(x,y)\in R_1^r,\\
\left(\frac{(r+2)x+2y}{r+4},\frac{2x+(r+2)y}{r+4}\right),&\mbox{if }(x,y)\in R_2^r,\\
\left(\frac{rx}{r+2},0\right),&\mbox{if }(x,y)\in R_3^r,\\
\left(0,\frac{ry}{r+2}\right),&\mbox{if }(x,y)\in R_4^r,
\end{cases}$$
$$h_r(x,y)=\begin{cases}
\sqrt{\frac{r}{r+4}}|x+y|,&\mbox{if }(x,y)\in R_1^r,\\
\sqrt{\frac{r}{r+4}}|x-y|,&\mbox{if }(x,y)\in R_2^r,\\
\sqrt{\frac{2rx^2+r(r+2)y^2}{2(r+2)}},&\mbox{if }(x,y)\in R_3^r,\\
\sqrt{\frac{r(r+2)x^2+2ry^2}{2(r+2)}},&\mbox{if }(x,y)\in R_4^r,
\end{cases}$$
and $\lim\limits_{r\nearrow\infty}h_r=f,$ $\lim\limits_{r\searrow0}h_r=0.$
\end{ex}
\begin{proof}
Using the same notation introduced in the previous example, it is convenient to split the infimum expression as follows:
\begin{align*}
g_r(\bar{x},\bar{y})&=\inf\limits_{(x,y)\in\R^2}\left\{(|x|+|y|)^2+\frac{r}{2}\left[(x-\bar{x})^2+(y-\bar{y})^2\right]\right\}\\
&=\min\left[I_{x,y>0},I_{x,y<0},I_{x>0,y<0},I_{x<0,y>0},I_{x=0},I_{y=0}\right].
\end{align*}
We omit the remaining details, as the procedure is the same as that of Example \ref{ex:maxnorm}.
Figure \ref{fig:secondone} shows the partitioning of $\R^2$ in the case $r=1,$ and Figure \ref{fig:hsecondone} is the corresponding function $h_1.$
\begin{figure}
\begin{center}\includegraphics[scale=0.03]{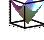}\end{center}
\caption{The four regions of the piecewise function $h_1(x,y).$}
\label{fig:secondone}
\end{figure}
\begin{figure}
\begin{center}\includegraphics[width=0.5\textwidth]{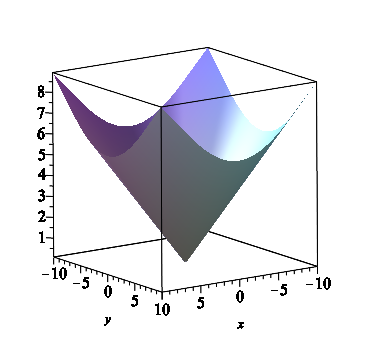}\end{center}
\caption{The function $h_1(x,y).$}
\label{fig:hsecondone}
\end{figure}
\noindent One can verify that $h_r\nearrow f,$ also by the same method as the previous example.
Figure \ref{fig:h_plus} shows the graphs of $h_r$ for several values of $r,$ and demonstrates the effect of $r\searrow0$ and $r\nearrow\infty.$
\begin{figure}
\begin{center}\includegraphics[width=0.5\textwidth]{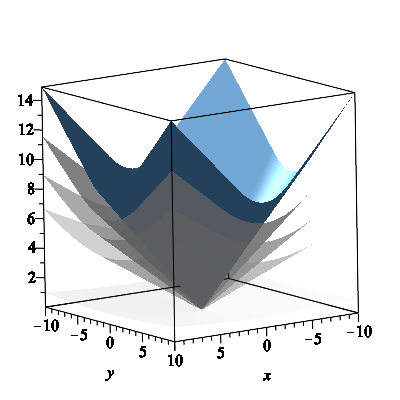}\end{center}
\caption{The function $h_r(x,y)$ from $r=0.01$ (grey) to $r=5$ (blue).}
\label{fig:h_plus}
\end{figure}
\noindent Finally, we consider the unit circle as a function of $r.$ Defining \begin{align*}\rho_r=&\{(x,y):h_r(x,y)=1\}\\
=&\left\{(x,y)\in R_1^r:\sqrt{\frac{r}{r+4}}|x+y|=1\right\}\cup\left\{(x,y)\in R_2^r:\sqrt{\frac{r}{r+4}}|x-y|=1\right\}\cup\\
&\left\{(x,y)\in R_3^r:\sqrt{\frac{2rx^2+r(r+2)y^2}{2(r+2)}}=1\right\}\cup\\
&\left\{(x,y)\in R_4^r:\sqrt{\frac{r(r+2)x^2+2ry^2}{2(r+2)}}=1\right\},
\end{align*}
we observe that as $r\nearrow\infty$ we recover the unit circle of the $1$-norm, $\{(x,y):|x+y|=1\}.$ Figure \ref{fig:rho_plus} displays $\rho_r$ for several values of $r.$\end{proof}
\begin{figure}
\begin{center}\includegraphics[scale=0.25]{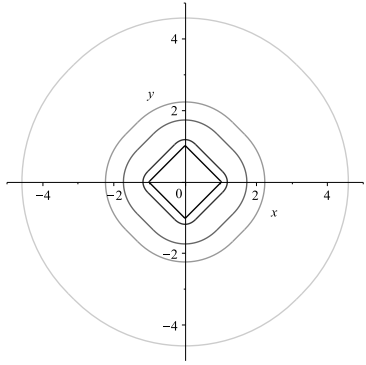}\end{center}
\caption{The unit circle $\rho_r$ for $r=1/10$ (light grey), $r=1/2,$ $r=1,$ $r=5,$ $r=100$ (black).}
\label{fig:rho_plus}
\end{figure}
\noindent The preceding examples, making use of Corollary \ref{normcor}, demonstrate the regularization power of the Moreau envelope; any norm can be converted into a norm that is smooth everywhere except at one point. Other envelope functions do not have this effect, as Example \ref{ex:paschnotsmooth} shows.
\begin{Fact}\label{fact:equiv}\emph{\cite[Exercise 9.12]{rockwets}}
If $f:\R^n\to\R$ is $L$-Lipschitz and $$h(x)=\inf_{y\in\R^n}\{f(y)+L\|y-x\|\}$$ for all $x\in\R^n,$ then $h\equiv f.$\end{Fact}
\begin{ex}\label{ex:paschnotsmooth}\emph{(Pasch-Hausdorff envelope)}
Let $f:\R^n\rightarrow\overline{\R}$ be a gauge function. Denote by $g_r$ the \emph{Pasch-Hausdorff envelope} of $f:$
$$g_r(x)=\inf\limits_{y\in\R^n}\{f(y)+r\|y-x\|\},$$
where $r>0.$ Then $g_r$ is a gauge, and $\partial g_r(x)=\partial f(p)\cap r\partial\|p-x\|,$ where $p$ is in the proximal mapping of $f$ at $x$. Moreover, when $f$ is an arbitrary norm,
$g_r$ is a norm. However, in that case $g_r$ is not necessarily differentiable
everywhere except at the origin, as is the case of \emph{Corollary \ref{normcor}}
where the Moreau envelope is used.
\end{ex}
\begin{proof}
To prove that $g_r$ is a gauge, we must show that
\begin{itemize}
\item[(i)] $g_r(x)\geq0$ for all $x\in\R^n,$ and $x=0\Rightarrow g_r(x)=0$,
\item[(ii)] $g_r(\alpha x)=\alpha g_r(x)$ for all $\alpha>0,$ and
\item[(iii)] $g_r$ is convex.
\end{itemize}
\noindent(i) Since $\min f(y)=0$ and $\min r\|y-x\|=0,$ we have $$\inf_{y\in\R^n}\{f(y)+r\|y-x\|\}\geq0~\forall x\in\R^n.$$
\noindent We have $g_r(0)=\inf\limits_{y\in\R^n}(f(y)+r\|y-0\|)=0,$ since both terms of the infimum are minimized at $y=0.$ Hence, $x=0\Rightarrow g_r(x)=0.$\medskip\\
\noindent(ii) Let $\alpha>0.$ Then, with $\tilde{y}=y/\alpha,$
\begin{align*}
g_r(\alpha x)&=\inf\limits_{y\in\R^n}\{f(y)+r\|y-\alpha x\|\}\\
&=\inf\limits_{\frac{y}{\alpha}\in\R^n}\left\{\alpha\left(f\left(\frac{y}{\alpha}\right)+
r\left\|\frac{y}{\alpha}-x\right\|\right)\right\}\\
&=\alpha\inf\limits_{\tilde{y}\in\R^n}\{f(\tilde{y})+r\|\tilde{y}-x\|\}=\alpha g_r(x).
\end{align*}
\noindent(iii) Since $(x,y)\mapsto f(y)+r\|y-x\|$ is convex, the marginal function $g_r$ is convex by \cite[Proposition 8.26]{convmono}.\medskip\\
Therefore, $g_r$ is a gauge. The expression for $\partial g_r$ comes from \cite[Proposition 16.48]{convmono}. Now let $f$ be a norm. To show that $g_r$ is a norm, we must show that
\begin{itemize}
\item[(iv)] $g_r(x)=0\Rightarrow x=0$,
\item[(v)] $g_{r}(-x)=g_{r}(x)$ for $x\in\R^n$, and
\item[(vi)] $g_r(x+y)\leq g_r(x)+g_r(y)$ for all $x,y\in\R^n.$
\end{itemize}
\noindent(iv) Let $g_r(x)=0$. Then there exists $\{y_k\}_{k=1}^\infty$ such that \begin{equation}\label{goesto0}f(y_k)+r\|y_k-x\|\to0.\end{equation}As$$0\leq f(y_k)\leq f(y_k)+r\|y_k-x\|\to0,$$ by the Squeeze Theorem we have $f(y_k)\to0$, and since $f$ is a norm, $y_k\to0$. Then by \eqref{goesto0} together with $f(y_k)\to0$, we have $\|y_k-x\|\to0$, i.e. $y_k\to x.$ Therefore, $x=0$.
\medskip\\
\noindent(v) As in Lemma~\ref{evenlem}, one can show that the Pasch-Hausdorff envelope of an even function is
even.
\medskip\\
\noindent(vi) By \cite[Theorem 4.7]{convanalrock}, it suffices that $g_r$ is convex.\medskip\\
Therefore, $g_r$ is a norm. To show that $g_r$ is not necessarily smooth everywhere except at one point, we consider a particular example. On $\R^2,$ define
$$f_1(x)=|x_1|+|x_2|,~f_2(x)=\sqrt{2}\sqrt{x_1^2+x_2^2}.$$
Then
$$g_{\sqrt{2}}(x)=\inf\limits_{y\in\R^n}\left\{f_1(y)+f_2(x-y)\right\}.$$ It is elementary to show that $f_1$ is $\sqrt{2}$-Lipschitz, so by Fact \ref{fact:equiv}, we have that $g_{\sqrt{2}}\equiv f_1$. Hence, $g_{\sqrt{2}}(x)=|x_1|+|x_2|,$ which is not smooth along the lines $x_1=0$ and $x_2=0.$
\end{proof}
\begin{rem}
Further work in this area could be done by replacing $q(x-y)=\frac{1}{2}\|x-y\|^2$ by a general distance function, for example the Bregman distance kernel:
$$D(x,y)=\begin{cases}
f(y)-f(x)-\langle\nabla f(x),y-x\rangle,&\mbox{if }y\in\dom f,x\in\intt\dom f,\\
\infty,&\mbox{otherwise.}
\end{cases}$$See \emph{\cite{chen2012themoreau,bregman}} for details on the Moreau envelope using the Bregman distance.
\end{rem}

\section{Conclusion}\label{sec:conc}

We established characterizations of Moreau envelopes: $e_rf$ is strictly convex if and only if $f$ is essentially strictly convex, and $f=e_rg$ with $g\in\Gamma_0(\R^n)$ if and only if $f^*$ is strongly convex with modulus $1/r.$ We saw differentiability properties of convex Moreau envelopes and used them to establish an explicit expression for the Moreau envelope of a piecewise cubic function. Finally, we presented a method for smoothing an arbitrary gauge function by applying the Moreau envelope, resulting in another norm function that is differentiable everywhere except on the kernel. A special application to an arbitrary norm function is presented.
\bibliographystyle{plain}
\bibliography{Bibliography}{}

\def\cprime{$'$}
\begin{thebibliography}{10}

\bibitem{aragon2007convergence}
F.~Arag\'on, A.~Dontchev, and M.~Geoffroy.
\newblock Convergence of the proximal point method for metrically regular
  mappings.
\newblock In {\em C{SVAA} 2004}, volume~17 of {\em ESAIM Proc.}, pages 1--8.
  EDP Sci., Les Ulis, 2007.

\bibitem{bajaj2017visualization}
A.~Bajaj, W.~Hare, and Y.~Lucet.
\newblock Visualization of the {$\varepsilon$}-subdifferential of piecewise
  linear--quadratic functions.
\newblock {\em Comput. Optim. Appl.}, 67(2):421--442, 2017.

\bibitem{convmono}
H.~Bauschke and P.~Combettes.
\newblock {\em Convex {A}nalysis and {M}onotone {O}perator {T}heory in
  {H}ilbert {S}paces}.
\newblock Springer, New York, 2011.

\bibitem{bayen2016about}
T.~Bayen and A.~Rapaport.
\newblock About {M}oreau--{Y}osida regularization of the minimal time crisis
  problem.
\newblock {\em J. Convex Anal.}, 23(1):263--290, 2016.

\bibitem{bento2014finite}
G.~Bento and J.~Cruz.
\newblock Finite termination of the proximal point method for convex functions
  on {H}adamard manifolds.
\newblock {\em Optimization}, 63(9):1281--1288, 2014.

\bibitem{chen2012themoreau}
Y.~Chen, C.~Kan, and W.~Song.
\newblock The {M}oreau envelope function and proximal mapping with respect to
  the {B}regman distances in {B}anach spaces.
\newblock {\em Vietnam J. Math.}, 40(2-3):181--199, 2012.

\bibitem{convhullalg}
B.~Gardiner and Y.~Lucet.
\newblock Convex hull algorithms for piecewise linear-quadratic functions in
  computational convex analysis.
\newblock {\em Set-Valued Var. Anal.}, 18(3-4):467--482, 2010.

\bibitem{hamilton1982inverse}
R.~Hamilton.
\newblock The inverse function theorem of {N}ash and {M}oser.
\newblock {\em Bull. Amer. Math. Soc.}, 7(1):65--222, 1982.

\bibitem{hare2014derivative}
W.~Hare and Y.~Lucet.
\newblock Derivative-free optimization via proximal point methods.
\newblock {\em J. Optim. Theory Appl.}, 160(1):204--220, 2014.

\bibitem{proxmap}
W.~Hare and R.~Poliquin.
\newblock Prox-regularity and stability of the proximal mapping.
\newblock {\em J. Convex Anal.}, 14(3):589--606, 2007.

\bibitem{hinter2009moreau}
M.~Hinterm\"uller and M.~Hinze.
\newblock Moreau--{Y}osida regularization in state constrained elliptic control
  problems: error estimates and parameter adjustment.
\newblock {\em SIAM J. Numer. Anal.}, 47(3):1666--1683, 2009.

\bibitem{diffprop}
A.~Jourani, L.~Thibault, and D.~Zagrodny.
\newblock Differential properties of the {M}oreau envelope.
\newblock {\em J. Funct. Anal.}, 266(3):1185--1237, 2014.

\bibitem{jourani2017moreau}
A.~Jourani and E.~Vilches.
\newblock Moreau--{Y}osida regularization of state-dependent sweeping processes
  with nonregular sets.
\newblock {\em J. Optim. Theory Appl.}, 173(1):91--116, 2017.

\bibitem{bregman}
C.~Kan and W.~Song.
\newblock The {M}oreau envelope function and proximal mapping in the sense of
  the {B}regman distance.
\newblock {\em Nonlinear Anal.}, 75(3):1385--1399, 2012.

\bibitem{keuthen2015moreau}
M.~Keuthen and M.~Ulbrich.
\newblock Moreau--{Y}osida regularization in shape optimization with geometric
  constraints.
\newblock {\em Comput. Optim. Appl.}, 62(1):181--216, 2015.

\bibitem{practasp}
C.~Lemar{\'e}chal and C.~Sagastiz{\'a}bal.
\newblock Practical aspects of the {M}oreau--{Y}osida regularization:
  theoretical preliminaries.
\newblock {\em SIAM J. Optim.}, 7(2):367--385, 1997.

\bibitem{plqmodel}
Y.~Lucet, H.~Bauschke, and M.~Trienis.
\newblock The piecewise linear-quadratic model for computational convex
  analysis.
\newblock {\em Comput. Optim. Appl.}, 43(1):95--118, 2009.

\bibitem{meng2001piecewise}
F.~Meng and Y.~Hao.
\newblock Piecewise smoothness for {M}oreau--{Y}osida approximation to a
  piecewise {$C^2$} convex function.
\newblock {\em Adv. Math.}, 30(4):354--358, 2001.

\bibitem{mifflin1999properties}
R.~Mifflin, L.~Qi, and D.~Sun.
\newblock Properties of the {M}oreau--{Y}osida regularization of a piecewise
  {$C^2$} convex function.
\newblock {\em Math. Program.}, 84(2):269--281, 1999.

\bibitem{moreau1963}
J.-J. Moreau.
\newblock Propri\'et\'es des applications ``prox''.
\newblock {\em C. R. Acad. Sci. Paris}, 256:1069--1071, 1963.

\bibitem{proximite}
J.-J. Moreau.
\newblock Proximit\'e et dualit\'e dans un espace {H}ilbertien.
\newblock {\em Bull. Soc. Math. France}, 93:273--299, 1965.

\bibitem{planwang2016}
C.~Planiden and X.~Wang.
\newblock Strongly convex functions, {M}oreau envelopes and the generic nature
  of convex functions with strong minimzers.
\newblock {\em SIAM J. Optim.}, 26(2):1341--1364, 2016.

\bibitem{genhess}
R.~Poliquin and R.~Rockafellar.
\newblock Generalized {H}essian properties of regularized nonsmooth functions.
\newblock {\em SIAM J. Optim.}, 6(4):1121--1137, 1996.

\bibitem{convanalrock}
R.~Rockafellar.
\newblock {\em Convex {A}nalysis}.
\newblock Princeton Landmarks in Mathematics. Princeton University Press,
  Princeton, NJ, 1997.

\bibitem{rockwets}
R.~Rockafellar and J.-B. Wets.
\newblock {\em Variational {A}nalysis}.
\newblock Springer-Verlag, Berlin, 1998.

\bibitem{wangklee}
X.~Wang.
\newblock On {C}hebyshev functions and {K}lee functions.
\newblock {\em J. Math. Anal. Appl.}, 368(1):293--310, 2010.

\bibitem{xiao2014proximal}
H.~Xiao and X.~Zeng.
\newblock A proximal point method for the sum of maximal monotone operators.
\newblock {\em Math. Methods Appl. Sci.}, 37(17):2638--2650, 2014.

\bibitem{funcanal}
K.~Yosida.
\newblock {\em Functional {A}nalysis}.
\newblock Die Grundlehren der Mathematischen Wissenschaften. Springer-Verlag,
  Berlin, 1965.

\bibitem{zaslavski2010convergence}
A.~Zaslavski.
\newblock Convergence of a proximal point method in the presence of
  computational errors in {H}ilbert spaces.
\newblock {\em SIAM J. Optim.}, 20(5):2413--2421, 2010.

\end{thebibliography}
\subsection*{Acknowledgement}
The authors thank the anonymous referee for the many useful comments and suggestions made to improve this manuscript. Chayne Planiden was supported by UBC University Graduate Fellowship and by Natural Sciences and Engineering Research Council of Canada. Xianfu Wang was partially supported by a Natural Sciences and Engineering Research Council of Canada Discovery Grant.
\end{document}